\documentclass[11pt,a4paper]{article}
\usepackage{times}
\usepackage[top=2.5cm,bottom=2.5cm,left=2.5cm,right=2.5cm]{geometry}
\usepackage{graphicx}
\usepackage{epstopdf}
\usepackage[usenames, dvipsnames]{color}
\usepackage{amsmath,amssymb,mathrsfs,amsthm}
\usepackage{enumerate}
\usepackage{subfigure}
\usepackage{float}
\usepackage{booktabs}
\usepackage[colorlinks,citecolor=red,linkcolor=blue]{hyperref}

\newtheorem{theorem}{Theorem}
\newtheorem{lemma}{Lemma}
\newtheorem{remark}{Remark}
\newtheorem{example}{Example}

\newtheorem{assumption}{Assumption}

\allowdisplaybreaks[4]

\numberwithin{equation}{section}
\numberwithin{theorem}{section}
\numberwithin{lemma}{section}
\numberwithin{example}{section}
\numberwithin{definition}{section}

\title{Crank-Nicolson schemes for sub-diffusion equations with nonsingular and singular source terms in time}
\author{Han Zhou\thanks{Institute of Mathematics, Hebei University of Technology, Tianjin 300401, China. This author was partially supported by the National Natural Science Foundation of China (No. 11901151). Email: zhouhan@hebut.edu.cn}
        \and
        Wenyi Tian\thanks{Corresponding author. Center for Applied Mathematics, Tianjin University, Tianjin 300072, China. This author was partially supported by the National Natural Science Foundation of China (No. 12071343). Email: twymath@gmail.com}
}
\date{}

\begin{document}
\maketitle
\begin{abstract}
  In this work, two Crank-Nicolson schemes without corrections are developed for sub-diffusion equations. First, we propose a Crank-Nicolson scheme without correction for problems with regularity assumptions only on the source term. Second, since the existing Crank-Nicolson schemes have a severe reduction of convergence order for solving sub-diffusion equations with singular source terms in time, we then extend our scheme and propose a new Crank-Nicolson scheme for problems with singular source terms in time. Second-order error estimates for both the two Crank-Nicolson schemes are rigorously established by a Laplace transform technique, which are numerically verified by some numerical examples.

  {\bf Keywords:} Sub-diffusion equation, singular source term, Crank-Nicolson scheme, Laplace transform, linear finite element

  {\bf AMS subject classifications:} 65M06, 65M60, 65M15, 35R11, 35R05
\end{abstract}

\section{Introduction}

We consider the following sub-diffusion equation with a singular source term in time and nonsmooth initial data,
\begin{equation}\label{eq:tfdesub}
  \left\{
  \begin{aligned}
    &{^{C\!}}D^{\alpha}_tu(x,t)+A u(x,t)=f(x,t), ~~&& (x,t)\in \Omega\times (0, T],\\
    &u(x,t)=0,~~&&(x,t)\in\partial\Omega\times (0, T],\\
    &u(x,0)=u^{0}(x),~~&&x\in \Omega.
  \end{aligned}\right.
\end{equation}
The operator $A$ denotes a self-adjoint positive definite second-order elliptic partial
differential operator in a bounded domain $\Omega\subset\mathbb{R}^d$ with boundary $\partial\Omega$, $d=1,2$, and the initial value $u^{0}(x)$ belongs to $L^{2}(\Omega)$.
The notation ${^{C\!}}D^{\alpha}_tu(x,t)$ with $0<\alpha<1$ is defined by
\begin{equation*}
  {^{C\!}}D^{\alpha}_tu(x,t)=\frac{1}{\Gamma(1-\alpha)}\int_0^t(t-\zeta)^{-\alpha}u'(x,\zeta)\mathrm{d}\zeta
\end{equation*}
refers to the $\alpha$-th order left Caputo derivative of function $u(x,t)$ with respect to variable $t$, where $\Gamma(\cdot)$ denotes the Gamma function given by $\displaystyle\Gamma(s)=\int_0^{\infty}t^{s-1}e^{-t}\mathrm{d}t$ for $s$ with real part $\Re(s)>0$.

The sub-diffusion equation \eqref{eq:tfdesub} was formulated in \cite{SchneiderW:1989} and then widely used to simulate anomalous diffusion phenomena in physics recently \cite{MetzlerK:2000}, where the the mean squared displacement of particle motion grows by sublinear rate in time.
Compared to the normal diffusion equations, the solutions of sub-diffusion equation \eqref{eq:tfdesub} and some other time-fractional evolution problems usually exhibit weakly singular property near the origin even if the given data are sufficiently smooth with respect to time \cite{BrunnerH:1986,MillerF:1971,SakamotoY:2011,StynesOG:2017}.

To solve this type of problems numerically, some efficient finite difference methods were developed, such as piecewise polynomial interpolation \cite{GaoSZ:2014,JinLZ:2016a,LinX:2007,StynesOG:2017,YanKF:2018,ZhangSW:2011} and convolution quadrature (CQ) \cite{CuestaLP:2006,JinLZ:2017,JinLZ:2018a,LubichST:1996,WangWY:2021,WangYYP:2020,ZengLLT:2015,ZhouT:2022}. Among the different discretization schemes in the literature, the CQ technique proposed in the pioneering work \cite{Lubich:1986b} by Lubich is flexible for designing high-order numerical schemes for approximating time-fractional evolution problems \cite{JinLZ:2017}. Due to the weakly singular property of the solutions near the origin, the direct application of CQ will lead to an order reduction to only first order in time, while the optimal convergence order can be preserved by a correction approach in \cite{CuestaLP:2006,JinLZ:2017}. This idea also was utilized in designing Crank-Nicolson CQ schemes for the sub-diffusion problem \eqref{eq:tfdesub} in \cite{JinLZ:2016a,WangWY:2021} to preserve the optimal second-order convergence rate. The Crank-Nicolson CQ scheme developed in \cite{JinLZ:2016a} needs corrections at two starting time steps. Furthermore, \cite{WangWY:2021} designed another Crank-Nicolson CQ scheme with only single-step initial correction.

For the source term in the sub-diffusion problem \eqref{eq:tfdesub} owning sufficient regularity in time, the schemes based on CQ can achieve optimal convergence order by some proper corrections. However, it was mentioned in \cite{ZhouT:2022} that the correction approach in the literature is not applicable to problem \eqref{eq:tfdesub} with source terms being singular at $t=0$ since $f(0)$ tends to infinity, such as $f(x,t)=t^\mu g(x)$ with $-1<\mu<0$. Thus the existing time-stepping schemes, including the Crank-Nicolson CQ schemes in \cite{JinLZ:2016a,WangWY:2021}, lost their optimal accuracy and have severe reduction of convergence order far below one, and such performance was observed in the numerical results by the BDF1- and BDF2-CQs in \cite{ZhouT:2022} and numerical examples in this paper.
Overall, singular source terms in problem \eqref{eq:tfdesub} bring new difficulties both in designing efficient time-stepping schemes and analyzing error bounds.
In \cite{ZhouT:2022}, two new time-stepping schemes based on BDF1 and BDF2 were proposed for \eqref{eq:tfdesub} with a class of source terms mildly singular in time, which can restore the optimal first and second convergence order, respectively, even for singular source term $f(x,t)=t^{\mu}g(x)$ with $-1<\mu<0$. Additionally, \cite{ZhouT:2022} also proposed a new analysis technique based on the Laplace transform instead of generating function of the source term  to establish error estimates of the proposed schemes. In \cite{ChenSZ:2022x}, the error estimates of the schemes based on BDF2 were analyzed by the discrete Laplace transform technique for \eqref{eq:tfdesub} with some singular source terms.

In this work, we concentrate on designing novel Crank-Nicolson schemes for the sub-diffusion problem \eqref{eq:tfdesub} with both nonsingular and singular source terms and analyzing their error estimates by developing the Laplace transform technique mentioned in \cite{ZhouT:2022}. As mentioned above, the existing Crank-Nicolson schemes \cite{JinLZ:2016a,WangWY:2021} both require corrections at starting time steps. Then our first objective is to design a novel Crank-Nicolson scheme without corrections for the sub-diffusion problem \eqref{eq:tfdesub}, which can also keep the optimal second-order convergence rate for source terms with low regularity. The second objective of this work is to develop a second-order Crank-Nicolson scheme for the problem \eqref{eq:tfdesub} with singular source terms, such as $f(x,t)=t^{\mu}g(x,t)$ with $-1<\mu<0$.

The rest of this paper is organized as follows. In Section \ref{sec:2}, we present some preliminary results on the sub-diffusion problem \eqref{eq:tfdesub} with singular source terms with respect to time.
In Section \ref{sec:3}, a novel fractional Crank-Nicolson scheme \eqref{eq:CN} without correction is proposed for the sub-diffusion problem \eqref{eq:tfdesub}, and the second-order error estimates are analyzed for nonsingular source terms. In Section \ref{sec:4}, we further design another fractional Crank-Nicolson scheme \eqref{eq:nCN} for singular source terms. The optimal second-order convergence rate is also rigorously established. In Section \ref{sec:5}, some numerical results are illustrated to show the effectiveness of the proposed Crank-Nicolson schemes and verify the theoretical convergence results.

\section{Preliminary} \label{sec:2}
The well-posedness and regularity of problem \eqref{eq:tfdesub} have been well established in \cite{Bajlekova:2001,SakamotoY:2011} for $f(x,t)\in L^{p}(0,T;L^{2}(\Omega))$ with $p>1$.
For the case that $f(x,t)$ belongs to the space $L^1(0,T;L^{2}(\Omega))$ and owns lower regularity at $t=0$ , the existence, uniqueness and regularity of the solution of \eqref{eq:tfdesub} were discussed in \cite{ZhouT:2022} as well.
The result is stated in the following theorem. For convenience of notation, $(\cdot,\cdot)$ denotes the inner product in $L^{2}(\Omega)$, and $\|\cdot\|$ denotes the corresponding norm throughout this paper.
\begin{theorem}[\cite{ZhouT:2022}]\label{thm:wpr}
  Let $u^{0}(x)\equiv0$ and $f(x,t)$ in \eqref{eq:tfdesub} satisfy Assumption \ref{asm:af} (in Section \ref{sec:4.2}). Then the problem \eqref{eq:tfdesub} has a unique solution $u\in C((0, T]; L^{2}(\Omega))$, which satisfies
  \begin{equation}\label{eq:2.3}
    \|u(t)\|\leq ct^{\alpha+\mu}, \quad \mu>-1,~~t>0.
  \end{equation}
\end{theorem}
For homogenous case of \eqref{eq:tfdesub} with $u^{0}\in L^{2}(\Omega)$, the corresponding result can be referred to \cite{SakamotoY:2011}.
In such case, there exists a unique weak solution $u\in C([0, T]; L^{2}(\Omega))$ to problem \eqref{eq:tfdesub} such that
  \begin{equation*}
    \max_{0\leq t\leq T}\|u(t)\|\leq c\|u^{0}\|.
  \end{equation*}

By Laplace transform approach, the solution of \eqref{eq:tfdesub} can be represented as
  \begin{equation}\label{eq:tildeu0}
    u(t)=\frac{1}{2\pi i}\int_{\Gamma}e^{s t}(s^{\alpha}+A)^{-1}\big(s^{\alpha-1}u^0+\hat{f}(s)\big)\mathrm{d}s,
  \end{equation}
  or
  \begin{equation}\label{eq:tildeu}
    u(t)=\bar{E}(t)u^0+\int_{0}^{t}E(t-s)f(s)\mathrm{d}s
  \end{equation}
  with operators $\bar{E}(\cdot)$ and $E(\cdot)$ on $L^{2}(\Omega)$  defined by
\begin{align}
  \bar{E}(t)=\frac{1}{2\pi i}\int_{\Gamma}e^{st}(s^{\alpha}+A)^{-1}s^{\alpha-1}\mathrm{d}s,\quad
  E(t)=\frac{1}{2\pi i}\int_{\Gamma}e^{st}(s^{\alpha}+A)^{-1}\mathrm{d}s,\label{eq:E}
\end{align}
  where
  \begin{equation}\label{gammac}
    \Gamma=\{\sigma+iy:\sigma>0,~ y\in\mathbb{R}\}.
  \end{equation}
It is known in \cite{LubichST:1996,Thomee:2006} that the resolvent of the symmetric elliptic operator $A$ obeys the following estimate
\begin{equation}\label{eq:resolv}
  \|(s+A)^{-1}\|\le M|s|^{-1},\quad \forall~s\in\Sigma_{\theta}
\end{equation}
for some $\theta\in (\pi/2,\pi)$, where $\Sigma_{\theta}$ is a sector of the
complex plane $\mathbb{C}$ and given by
\begin{equation}\label{eq:sigt}
  \Sigma_{\theta}=\big\{z\in\mathbb{C}\setminus\{0\}:|\mathrm{arg}z|<\theta\big\}.
\end{equation}
Therefore by the resolvent estimate \eqref{eq:resolv} and Cauchy's theorem, $\Gamma$ in \eqref{eq:tildeu0} and \eqref{eq:E} can be replaced by $\Gamma_{\varepsilon}^{\theta}\cup S_{\varepsilon}$ defined as
  \begin{equation}\label{eq:gammavts}
    \Gamma_{\varepsilon}^{\theta}\cup S_{\varepsilon}=\{\rho e^{\pm i\theta}:~ \rho\geq\varepsilon\}\cup\{\varepsilon e^{i\xi}: -\theta\leq \xi\leq \theta\}.
  \end{equation}

The semidiscrete problem by finite element method for \eqref{eq:tfdesub} is to find $u_{h}(t)\in X_{h}$ satisfying
\begin{equation}\label{eq:sFE}
  \begin{aligned}
    \big({^{C\!}}D^{\alpha}_tu_{h}(t),\varphi\big)+A\big(u_{h}(t),\varphi\big)&=\big(f(t),\varphi\big),\quad \forall~ \varphi\in X_{h},\\
    u_h(0)&=P_hu^0,
  \end{aligned}
\end{equation}
where $X_{h}\subset H_0^1(\Omega)$ is a continuous piecewise linear finite element space on a regular triangulation mesh $\mathcal{T}_h$ of the domain $\Omega$, $h=\max_{T\in\mathcal{T}_h}\mathrm{diam}(T)$ is the maximal diameter. $A(\cdot,\cdot)$ denotes the bilinear form associated with the elliptic operator $A$. The $L^{2}$-projection operator $P_{h}:L^{2}(\Omega)\to X_{h}$ in \eqref{eq:sFE} is defined by
$$(P_{h}\varphi, \psi)=(\varphi, \psi), ~\forall~\psi\in X_{h}.$$
Furthermore, the semidiscrete scheme \eqref{eq:sFE} can be rewritten in the form of
\begin{equation}\label{eq:sFEO}
  {^{C\!}}D^{\alpha}_tu_{h}(t)+A_{h} u_{h}(t)=f_{h}(t),~\forall~t>0,~\text{and}~u_{h}(0)=P_{h}u^{0}
\end{equation}
with $f_{h}=P_{h}f$, where the operator $A_{h}: X_{h}\rightarrow X_{h}$ is defined by
\begin{equation}\label{eq:Deltah}
  (A_h\varphi, \psi)=A(\varphi,\psi),~\forall~\varphi, \psi\in X_{h}.
\end{equation}

Similarly, the semidiscrete solution of \eqref{eq:sFEO} for $t>0$ can be represented by
\begin{equation}\label{eq:uh}
  u_{h}(t)=\frac{1}{2\pi i}\int_{\Gamma_{\varepsilon}^{\theta}\cup S_{\varepsilon}}
  e^{st}(s^{\alpha}+A_h)^{-1}\big(s^{\alpha-1}u_h(0)+\hat{f}_{h}(s)\big)\mathrm{d}s,
\end{equation}
or
\begin{equation}\label{eq:uhE}
  u_{h}(t)=\bar{E}_h(t)u_h(0)+\int_0^tE_h(t-s)f_h(s)\mathrm{d}s,
\end{equation}
where $\bar{E}_h(\cdot)$ and $E_h(\cdot)$ are operators on $X_{h}$ given by
\begin{align}
  \bar{E}_h(t)&=\frac{1}{2\pi i}\int_{\Gamma_{\varepsilon}^{\theta}\cup S_{\varepsilon}}e^{st}s^{\alpha-1}(s^{\alpha}+A_{h})^{-1}\mathrm{d}s,\label{eq:bEh}\\
  E_h(t)&=\frac{1}{2\pi i}\int_{\Gamma_{\varepsilon}^{\theta}\cup S_{\varepsilon}}e^{st}(s^{\alpha}+A_{h})^{-1}\mathrm{d}s.\label{eq:Eh}
\end{align}

The error estimates for the semidiscrete Galerkin finite element scheme \eqref{eq:sFEO} have been well established in \cite{JinLPZ:2015,JinLZ:2013,Karaa2018,ZhouT:2022}. Moreover, the error analysis on the lumped mass finite element scheme has been discussed in \cite{JinLPZ:2015,JinLZ:2013,ZhouT:2022} as well.

\section{Crank-Nicolson scheme for nonsingular source terms}\label{sec:3}
In this section, we propose a novel fully discrete Crank-Nicolson scheme without corrections for solving  \eqref{eq:tfdesub} and establish the temporal error estimates. The discussion is based on the spatial semidiscrete scheme \eqref{eq:sFEO} by the Galerkin finite element method.

\subsection{Crank-Nicolson scheme I}\label{sec:CN}

We first reformulate the semidiscrete scheme \eqref{eq:sFEO} by introducing two functions $U_{h}(t)$ and $F_{h}(t)$, which are defined by $\int_{0}^{t}u_{h}(\xi)\mathrm{d}\xi$ and $\int_{0}^{t}f_{h}(\xi)\mathrm{d}\xi$, respectively. If $u_{h}$ and $f_{h}$ are in $L^{1}(0,T;X_{h})$, then $U_{h}(t)$ and $F_{h}(t)$ belong to the space $C([0,T];X_{h})$ and satisfy
\begin{equation}\label{eq:u1}
  D_t{U}_{h}(t)={u}_{h}(t),\quad \quad {U}_{h}(0)=0
\end{equation}
and
\begin{equation}\label{eq:f1}
  D_tF_{h}(t)=f_{h}(t),\quad \quad  F_{h}(0)=0
\end{equation}
for a.e. $t>0$, where $D_t:=\partial /\partial t$. Next, substituting \eqref{eq:u1} and \eqref{eq:f1} into \eqref{eq:sFEO} yields
\begin{equation}\label{eq:DU}
  {^{C\!}}D^{\alpha}_tD_tU_{h}(t)+A_{h}D_tU_{h}(t)=D_tF_{h}(t).
\end{equation}
Moreover, integrating \eqref{eq:DU} from $0$ to $t$ and using the semigroup property of fractional integrals, we obtain
\begin{equation}\label{eq:daU}
  {^{C\!}}D^{\alpha}_tU_{h}+A_{h}U_{h}=F_{h}+D_t\frac{t^{2-\alpha}}{\Gamma(3-\alpha)}u_{h}(0).
\end{equation}
Therefore, \eqref{eq:daU} together with \eqref{eq:u1} and \eqref{eq:f1}, can be recognized as an equivalent form of \eqref{eq:sFEO}. In what follows, a novel fractional Crank-Nicolson scheme will be proposed based on the equivalent form.

Let $\tilde{U}_{h}(t)$ and $\tilde{u}_{h}(t)$ be approximations to $U_{h}(t)$ and $u_{h}(t)$ in \eqref{eq:daU} and \eqref{eq:u1}, respectively. The notation $D_{\tau}^{\alpha}$ denotes the Gr\"unwald-Letnikov difference operator defined by
\begin{equation}\label{eq:cCN}
  D_{\tau}^{\alpha}\tilde{U}_{h}(t)=\tau^{-\alpha}\sum_{j=0}^{\infty}\sigma_{j}\tilde{U}_{h}(t-j\tau),
\end{equation}
where $\{\sigma_{j},~j\ge 0\}$ are coefficients of a generating function such that
\begin{equation}\label{eq:bdf1}
  \frac{1}{\tau^\alpha}\sum\limits_{j=0}^{\infty}\sigma_{j}z^{j}=\omega_1(z)^\alpha,\quad \text{with}~~ \omega_1(z)=\frac{1}{\tau}(1-z),
\end{equation}
and
\begin{equation}\label{eq:sigma}
  \sigma_0=1,~~\sigma_j=(1-\frac{\alpha+1}{j})\sigma_{j-1},~~j\ge1.
\end{equation}
The equation \eqref{eq:daU} is approximated at $t-\frac{\alpha}{2}\tau$ by fractional Crank-Nicolson approach \cite{Dimitrov:2014,JinLZ:2018a}, as $U_{h}(0)=0$, that is
\begin{align*}
  D_{\tau}^{\alpha}U_{h}(t)&={^{C\!}}D^{\alpha}_tU_{h}(t-\frac{\alpha}{2}\tau)+O(\tau^2)\\
  &=(1-\frac{\alpha}{2})~{^{C\!}}D^{\alpha}_tU_{h}(t)+\frac{\alpha}{2}~{^{C\!}}D^{\alpha}_tU_{h}(t-\tau)+O(\tau^2).
\end{align*}
Then
$\tilde{U}_{h}(t)$ and $\tilde{u}_{h}(t)$ satisfy the following difference equations
\begin{align}
    &D_{\tau}^{\alpha}\tilde{U}_{h}(t)+(1-\frac{\alpha}{2})A_{h}\tilde{U}_{h}(t)
    +\frac{\alpha}{2}A_{h}\tilde{U}_{h}(t-\tau)=(1-\frac{\alpha}{2})F_{h}(t)
    +\frac{\alpha}{2}F_{h}(t-\tau)\notag\\
    &\hskip3.7cm+(1-\frac{\alpha}{2})D_{\tau}\frac{t^{2-\alpha}}{\Gamma(3-\alpha)}u_{h}(0)
    +\frac{\alpha}{2}D_{\tau}\frac{(t-\tau)^{2-\alpha}}{\Gamma(3-\alpha)}u_{h}(0),\label{eq:DECN}\\
    &\tilde{u}_{h}(t)=D_{\tau}\tilde{U}_{h}(t):=\frac{1}{\tau}\Big(\frac{3}{2}
    \tilde{U}_{h}(t)-2\tilde{U}_{h}(t-\tau)+\frac{1}{2}\tilde{U}_{h}(t-2\tau)\Big)\label{eq:cuCN}\end{align}
for $t>0$, and are prescribed by zero for $t\leq 0$. Here $D_{\tau}$ denotes the second-order BDF operator.

Given a uniform partition of the interval $[0, T]$,
\begin{equation*}
  0=t_{0}<t_{1}<\cdots<t_{N-1}<t_{N}=T.
\end{equation*}
The step size of the uniform mesh is denoted by $\tau=T/N$ and $t_{n}=n\tau$ for $0\le n\le N$.
Then choosing $t=t_{n}$ for $n=1,\cdots, N$ in \eqref{eq:DECN} and \eqref{eq:cuCN}, we propose a fractional Crank-Nicolson scheme for solving \eqref{eq:tfdesub} as follows
\begin{equation}\label{eq:CN}
  \left\{
  \begin{aligned}
    &\tau^{-\alpha}\sum_{j=0}^{n}\sigma_{j}\tilde{U}_{h}^{n-j}+(1-\frac{\alpha}{2})A_{h}\tilde{U}_{h}^{n}
    +\frac{\alpha}{2}A_{h}\tilde{U}_{h}^{n-1}
    =(1-\frac{\alpha}{2})F_{h}^{n}
    +\frac{\alpha}{2}F_{h}^{n-1}\\
    &\hskip3cm+(1-\frac{\alpha}{2})D_{\tau}\frac{t_{n}^{2-\alpha}}{\Gamma(3-\alpha)}u_{h}(0)
     +\frac{\alpha}{2}D_{\tau}\frac{t_{n-1}^{2-\alpha}}{\Gamma(3-\alpha)}u_{h}(0),\\
    &\tilde{u}_{h}^{n}=\frac{1}{\tau}\Big(\frac{3}{2}\tilde{U}_{h}^{n}-2\tilde{U}_{h}^{n-1}
    +\frac{1}{2}\tilde{U}_{h}^{n-2}\Big),
  \end{aligned}\right.\tag{CN-I}
\end{equation}
with $\tilde{U}_{h}^{0}=0$, where $\tilde{u}_{h}^{n}:=\tilde{u}_{h}(t_{n})$, $\tilde{U}_{h}^{n}:=\tilde{U}_{h}(t_{n})$
and $F_{h}^{n}:=F_{h}(t_{n})$.

\subsection{Solution representations}
By taking Laplace transform on \eqref{eq:u1} and \eqref{eq:daU}, the semidiscrete solution $u_{h}(t)$ in \eqref{eq:sFEO} can also be rewritten as
\begin{equation}\label{eq:uh1}
  u_{h}(t)=\frac{1}{2\pi i}\int_{\Gamma_{\varepsilon}^{\theta}\cup S_{\varepsilon}}e^{st}s\hat{U}_{h}(s)\mathrm{d}s,
\end{equation}
where $\Gamma_{\varepsilon}^{\theta}\cup S_{\varepsilon}$ is defined by \eqref{eq:gammavts} and
\begin{equation}\label{eq:hatUh}
  \hat{U}_{h}(s)=(s^{\alpha}+A_{h})^{-1}
  \big(s^{\alpha-2}u_{h}(0)+s^{-1}\hat{f}_{h}(s)\big).
\end{equation}
The discrete operator $A_h$ defined by \eqref{eq:Deltah} also satisfies the resolvent estimate $\|(s^{\alpha}+A_{h})^{-1}\|\le c|s|^{-\alpha}$ for $s\in\Gamma_{\varepsilon}^{\theta}\cup S_{\varepsilon}$.
Then it follows from \eqref{eq:hatUh} that
\begin{equation}\label{eq:hUhs}
  \|\hat{U}_{h}(s)\|\le c|s|^{-\alpha}(|s|^{\alpha-2}\|u^0\|+|s|^{-1}\|\hat{f}_{h}(s)\|\big).
\end{equation}

It is indicated in \cite{ZhouT:2022} that $\tilde{u}_{h}(t)$ in \eqref{eq:cuCN} is continuous for $t>0$, and can be represented by
\begin{equation}\label{eq:tuh}
  \tilde{u}_{h}(t)=\frac{1}{2\pi i}\int_{\Gamma}e^{st}\omega_2(e^{-s\tau})\widehat{\tilde{U}_{h}}(s)\mathrm{d}s,
\end{equation}
where $\Gamma$ is given by \eqref{gammac} and
\begin{equation}\label{eq:bdf2}
  \omega_{2}(z)=\tau^{-1}\big(\frac{3}{2}-2z+\frac{1}{2}z^{2}\big).
\end{equation}
Moreover, with $\tilde{U}_{h}(t)=0$ for $t\leq 0$, we have from \eqref{eq:cCN} that the Laplace transform of $D_{\tau}^{\alpha}\tilde{U}_{h}(t)$ equals to
\begin{equation*}
  \widehat{D_{\tau}^{\alpha}\tilde{U}_{h}}(s)=\tau^{-\alpha}
  \sum_{j=0}^{\infty}\sigma_{j}\int_{0}^{+\infty}e^{-s t}\tilde{U}_{h}(t-j\tau)\mathrm{d}t
  =\omega_{1}(e^{-s\tau})^{\alpha}\widehat{\tilde{U}_{h}}(s).
\end{equation*}
Then it follows from \eqref{eq:DECN} that
\begin{equation}\label{eq:bUh}
  \widehat{\tilde{U}_{h}}(s)=\big(\omega(e^{-s\tau})^{\alpha}+A_{h}\big)^{-1}
  \big(s^{-1}\hat{f}_{h}(s)+\omega_2(e^{-s\tau})s^{\alpha-3}u_{h}(0)\big),
\end{equation}
where
\begin{equation}\label{eq:gfCN}
  \omega(z)=\frac{1-z}{\tau(1-\frac{\alpha}{2}+\frac{\alpha}{2}z)^{1/\alpha}}.
\end{equation}
In addition, \eqref{eq:tuh} can also be rewritten as
\begin{equation}\label{eq:buhE}
  \tilde{u}_{h}(t)=\bar{E}_h^{\tau}(t)u_h(0)+\int_0^tE_h^{\tau}(t-\zeta)f_h(\zeta)\mathrm{d}\zeta,
\end{equation}
where $\bar{E}_h^{\tau}(\cdot)$ and $E_h^{\tau}(\cdot)$ are operators on $X_{h}$ given by
\begin{align}
  \bar{E}_h^{\tau}(t)&=\frac{1}{2\pi i}\int_{\Gamma}e^{st}\omega_2(e^{-s\tau})^2(\omega(e^{-s\tau})^{\alpha}+A_{h})^{-1}s^{\alpha-3}\mathrm{d}s,\label{eq:bEht}\\
  E_h^{\tau}(t)&=\frac{1}{2\pi i}\int_{\Gamma}e^{st}\omega_2(e^{-s\tau})(\omega(e^{-s\tau})^{\alpha}+A_{h})^{-1}s^{-1}\mathrm{d}s.\label{eq:Eht}
\end{align}

From Lemmas B.3 in \cite{ZhouT:2022}, $\omega_2(e^{-s\tau})$ in \eqref{eq:bdf2} satisfies the following preliminary lemma for the error analysis of the fractional Crank-Nicolson scheme \eqref{eq:CN}.

\begin{lemma} \label{lem:fbdf2err}
  If $s\in\mathbb{C}$ and $|s\tau|\le r$ for finite $r>0$, then
  \begin{equation}
    |\omega_2(e^{-s\tau})|\leq C|s| \text{~~and~~}
    |s^{\beta}-\omega_2(e^{-s\tau})^{\beta}|\leq C\tau^2|s|^{\beta+2}
  \end{equation}
  hold for $0<\beta\leq 1$, where $C$ denotes a generic constant dependent on the radius $r$.
\end{lemma}

In addition, the term $\omega(e^{-s\tau})$ in \eqref{eq:gfCN} satisfies the following two preliminary lemmas, which are obtained from the results of Lemmas 3.3 and 3.4 in \cite{JinLZ:2018a}.

\begin{lemma}\label{lem:fCN}
  Let $\alpha\in(0,1)$ and $\phi\in(\alpha\pi/2, \pi)$ be fixed. Then there exists a $\delta_0 > 0$ (independent of $\tau$) such that for $\delta\in(0, \delta_0]$ and $\theta\in(\pi/2, \pi/2+\delta_0]$, we have $\omega(e^{-s\tau})^{\alpha}\in\Sigma_{\phi}$ for any $s\in\Sigma_{\pi/2}\cup\{z\in\Sigma_{\theta}\setminus\Sigma_{\pi/2}:~|\Im z|\le\pi/\tau\}$.
\end{lemma}

\begin{lemma}\label{lem:fCNerr}
  Let $\alpha\in (0, 1)$. There exists a constant $\delta_1$, for $\delta\in(0,\delta_1]$ and $\theta\in(\pi/2, \pi/2+\delta_1]$, we have for any $s\in\{z\in\Sigma_{\theta}\setminus\Sigma_{\pi/2}:~|\Im z|\le\pi/\tau\}$ and $0<\beta\le 1$ that
  \begin{equation*}
    C_0|s|\le|\omega(e^{-s\tau})|\le C_1|s| \text{~~and~~}
    |s^{\beta}-\omega(e^{-s\tau})^{\beta}|\leq C\tau^2|s|^{\beta+2}.
  \end{equation*}
\end{lemma}
By \eqref{eq:bUh} and Lemma \ref{lem:fCN}, we further have
\begin{equation}\label{eq:htUhs}
  \|\widehat{\tilde{U}_{h}}(s)\|\le c|\omega(e^{-s\tau})|^{-\alpha}\big(|s|^{-1}\|\hat{f}_{h}(s)\|+|\omega_2(e^{-s\tau})|\cdot|s|^{\alpha-3}\|u^0\|\big)
\end{equation}
holds for any $s\in\Sigma_\theta\setminus\Sigma_{\pi/2}$ with $|\Im s|\le\pi/\tau$ or $s\in\Sigma_{\pi/2}$.
\begin{lemma}\label{lem:uhn}
  Let $\tilde{u}_{h}(t)$ be the solution to the difference equations \eqref{eq:DECN}-\eqref{eq:cuCN} with $u^0(x)\in L^2(\Omega)$ and $f_h(t)=t^{\mu}g_h(x)$, $\mu>-1$. Then \eqref{eq:tuh} becomes
      \begin{equation}\label{eq:lemuhn}
        \begin{split}
          \tilde{u}_{h}(t)&=\frac{1}{2\pi i}\int_{\Gamma_{\varepsilon,\tau}^{\theta}\cup S_{\varepsilon}}e^{s t}\omega_2(e^{-s\tau})\widehat{\tilde{U}_{h}}(s)\mathrm{d}s\\
          &~~~~~+\sum_{\substack{p=-\infty \\ p\ne 0}}^{+\infty}\frac{1}{2\pi i}\int_{\Gamma_{0,\tau}^{\theta}}e^{(s+i2p\pi/\tau) t}\omega_{2}(e^{-s\tau})\widehat{\tilde{U}_{h}}(s+i2p\pi/\tau)\mathrm{d}s
        \end{split}
      \end{equation}
  for $t\in(0,T]$, where $\widehat{\tilde{U}_{h}}(s)$, $\Gamma_{\varepsilon,\tau}^{\theta}$ and $S_{\varepsilon}$ are given by \eqref{eq:bUh}, \eqref{eq:gett} and \eqref{eq:gammavts}, respectively.
\end{lemma}
\begin{proof}
  The solution $\tilde{u}_{h}(t)$ to \eqref{eq:cuCN} can be reformulated from \eqref{eq:tuh} as follows
  \begin{equation}\label{eq:bbuh}
    \tilde{u}_{h}(t)=\lim_{L\to+\infty}\frac{1}{2\pi i}\int_{\sigma-iL}^{\sigma+iL}
    e^{st}\omega_2(e^{-s\tau})\widehat{\tilde{U}_{h}}(s)\mathrm{d}s,
  \end{equation}
  where $\sigma=t^{-1}$ when $t\ge\tau$ and $\sigma=\tau^{-1}$ when $0<t\le\tau$.
  Then for any $L>0$ and fixed $\tau>0$, there exists $\bar{N}\in\mathbb{N}^{+}$ such that $\left( 2\bar{N}+1\right)\pi/\tau\le L\le  (2\bar{N}+3)\pi/\tau$, and the integral in \eqref{eq:bbuh} can be divided into three parts
  \begin{equation}\label{eq:iut}\small
    \begin{split}
    \int_{\sigma-iL}^{\sigma+iL}e^{st}\omega_2(e^{-s\tau})\widehat{\tilde{U}_{h}}(s)\mathrm{d}s
         &=\Big(\int_{\sigma+i(2\bar{N}+1)\frac{\pi}{\tau}}^{\sigma+iL}
         +\int_{\sigma-i(2\bar{N}+1)\frac{\pi}{\tau}}^{\sigma+i(2\bar{N}+1)\frac{\pi}{\tau}}+\int_{\sigma-iL} ^{\sigma-i(2\bar{N}+1)\frac{\pi}{\tau}}\Big)
         e^{st}\omega_2(e^{-s\tau})\widehat{\tilde{U}_{h}}(s)\mathrm{d}s.
    \end{split}
  \end{equation}
In addition, together with $f_{h}(t)=t^{\mu}g_h(x)$ and \eqref{eq:htUhs}, we get
\begin{equation}\label{eq:Fs}
  \begin{aligned}
    \|\hat{f}_{h}(s)\|&\le c|s|^{-\mu-1}\|g_h\|,\\
    \|\widehat{\tilde{U}_{h}}(s)\|&\le c|\omega(e^{-s\tau})|^{-\alpha}\big(|s|^{-\mu-2}\|g\|+|\omega_2(e^{-s\tau})|\cdot|s|^{\alpha-3}\|u^0\|\big).
  \end{aligned}
\end{equation}
  For the first integral in \eqref{eq:iut}, it follows from \eqref{eq:Fs} that
  \begin{equation*}
    \begin{split}
      &\|\int_{\sigma+i(2\bar{N}+1)\pi/\tau}^{\sigma+iL}e^{st}\omega_2(e^{-s\tau})
      \widehat{\tilde{U}_{h}}(s)\mathrm{d}s\|\\
       &\leq c\int_{\sigma+i(2\bar{N}+1)\pi/\tau}^{ \sigma+i(2\bar{N}+3)\pi/\tau}|e^{st}||\omega_2(e^{-s\tau})|
       |\omega(e^{-s\tau})|^{-\alpha}
       \big(|s|^{-\mu-2}\|g_h\|+|\omega_2(e^{-s\tau})|\cdot|s|^{\alpha-3}\|u^0\|\big)|\mathrm{d}s|   \\
       &\leq c\int_{(2\bar{N}+1)\pi/\tau}^{(2\bar{N}+3)\pi/\tau}e^{\sigma t}\big((\sigma+y-2\bar{N}\pi/\tau)^{1-\alpha}y^{-\mu-2}\|g_h\|
       +(\sigma+y-2\bar{N}\pi/\tau)^{2-\alpha}y^{\alpha-3}\|u^0\|\big)\mathrm{d}y \\
       &\leq c\Big(\frac{\tau^{\alpha+\mu}\|g_h\|}{(2\bar{N}+1)^{\mu+2}}
       +\frac{\|u^0\|}{(2\bar{N}+1)^{3-\alpha}}\Big)e^{\sigma t}.
    \end{split}
  \end{equation*}
  Then the above bound tends to zero when $L\to+\infty$ ($\bar{N}\to+\infty$). The similar result can also be derived for the third integral in \eqref{eq:iut}.

  For the estimate on the second integral in \eqref{eq:iut}, we first introduce some integral curves as follows
  \begin{equation}\label{eq:gett}
    \Gamma_{\varepsilon,\tau}^{\theta}=\{\rho e^{\pm i\theta}:~ \varepsilon\leq\rho\leq \pi/(\tau\sin\theta)\},
  \end{equation}
  \begin{equation}
    \Gamma^{+}=\{\xi+i(2\bar{N}+1)\pi/\tau:~ \pi/\tau\cot\theta\le \xi\le \sigma\},
  \end{equation}
  \begin{equation}
    \Gamma^{-}=\{\xi-i(2\bar{N}+1)\pi/\tau:~ \pi/\tau\cot\theta\le \xi\le \sigma\}.
  \end{equation}
  As $\widehat{\tilde{U}_{h}}(s)$ in \eqref{eq:bUh} is analytic in the sector $\Sigma_{\theta}$, then using the Cauchy's theorem and the periodic property of exponential function we obtain that
  \begin{equation}\label{eq:812}
    \begin{split}
      &\int_{\sigma-i(2\bar{N}+1)\pi/\tau}^{\sigma+i(2\bar{N}+1)\pi/\tau}
      e^{s t}\omega_2(e^{-s\tau})\widehat{\tilde{U}_{h}}(s)\mathrm{d}s \\
      &=\int_{\Gamma^{-}\cup \Gamma^{+}}e^{s t}\omega_2(e^{-s\tau})\widehat{\tilde{U}_{h}}(s)\mathrm{d}s+\int_{\Gamma_{\varepsilon,\tau}^{\theta}\cup S_{\varepsilon}}e^{s t}\omega_2(e^{-s\tau})\widehat{\tilde{U}_{h}}(s)\mathrm{d}s   \\
      &~~~~+\sum_{\substack{p=-\bar{N} \\ p\ne 0}}^{\bar{N}}\int_{\Gamma_{0,\tau}^{\theta}}e^{(s+i2p\pi/\tau) t}\omega_2(e^{-s\tau})\widehat{\tilde{U}_{h}}(s+i2p\pi/\tau)\mathrm{d}s.   \\
    \end{split}
  \end{equation}
  The first term in the right hand side of \eqref{eq:812} can be estimated as follows
  \begin{align*}
      &\|\int_{\Gamma^{-}\cup\Gamma^{+}}e^{s t}\omega_2(e^{-s\tau})\widehat{\tilde{U}_{h}}(s)\mathrm{d}s\|\\
      &\leq c\Big(\int_{\Gamma^{-}}|e^{st}||\omega_2(e^{-(s+i2\bar{N}\pi/\tau)\tau})|
       |\omega(e^{-(s+i2\bar{N}\pi/\tau)\tau})|^{-\alpha}
       |s|^{-\mu-2}\|g_h\||\mathrm{d}s|\\
      &~~~~~+\int_{\Gamma^{-}}|e^{st}||\omega_2(e^{-(s+i2\bar{N}\pi/\tau)\tau})|^2
       |\omega(e^{-(s+i2\bar{N}\pi/\tau)\tau})|^{-\alpha}\cdot|s|^{\alpha-3}\|u^0\||\mathrm{d}s|\\
      &~~~~~+\int_{\Gamma^{+}}|e^{st}||\omega_2(e^{-(s-i2\bar{N}\pi/\tau)\tau})|
       |\omega(e^{-(s-i2\bar{N}\pi/\tau)\tau})|^{-\alpha}
       |s|^{-\mu-2}\|g_h\||\mathrm{d}s|\\
      &~~~~~+\int_{\Gamma^{+}}|e^{st}||\omega_2(e^{-(s-i2\bar{N}\pi/\tau)\tau})|^2
       |\omega(e^{-(s-i2\bar{N}\pi/\tau)\tau})|^{-\alpha}\cdot|s|^{\alpha-3}\|u^0\||\mathrm{d}s|\Big)\\
      &\leq c\int_{\pi/\tau\cot\theta}^{\sigma}e^{xt}(|\xi|+\pi/\tau)^{1-\alpha}
      |(2\bar{N}+1)\pi/\tau|^{-\mu-2}\|g_h\|\mathrm{d}\xi  \\
      &~~~~~+c\int_{\pi/\tau\cot\theta}^{\sigma}e^{xt}(|\xi|+\pi/\tau)^{2-\alpha}
      |(2\bar{N}+1)\pi/\tau|^{\alpha-3}\|u^0\|\mathrm{d}\xi\\
       &\leq c\Big(\frac{\tau^{\alpha+\mu}\|g_h\|}{(2\bar{N}+1)^{\mu+2}}
       +\frac{\|u^0\|}{(2\bar{N}+1)^{3-\alpha}}\Big)e^{\sigma t},
  \end{align*}
  which tends to zero for $L\to+\infty$ ($\bar{N}\to+\infty$). Therefore, the result \eqref{eq:lemuhn} is obtained from \eqref{eq:bbuh} and \eqref{eq:812}.
\end{proof}
\subsection{Error estimates for nonsingular source terms}
In this subsection, we establish the temporal discrete error estimates of scheme \eqref{eq:CN} by means of Laplace transform for the case of nonsingular source terms in time.
\begin{lemma}\label{lem:herr}
  Let $\bar{E}_{h}(\cdot)$ and $\bar{E}_h^{\tau}(\cdot)$ be given by \eqref{eq:bEh} and \eqref{eq:bEht}, respectively. Then there holds
  \begin{equation}\label{eq:81}
    \|(\bar{E}_h(t)-\bar{E}_h^{\tau}(t))P_hu^0\|\leq c\tau^2 t^{-2}\|u^0\|,\quad \quad \forall~t\in(0,T].
  \end{equation}
\end{lemma}
\begin{proof}
  Let $f_h(t)\equiv0$, it follows from \eqref{eq:uh1}, \eqref{eq:tuh} and Lemma \ref{lem:uhn} that
  \begin{equation}\label{eq:gerror}
    \begin{split}
      &(\bar{E}_h(t)-\bar{E}_h^{\tau}(t))P_hu^0\\
      &=\frac{1}{2\pi i}\int_{\Gamma_{\varepsilon}^{\theta}\backslash\Gamma_{\varepsilon,\tau}^{\theta}}e^{st}
      s\hat{U}_{h}(s)\mathrm{d}s+\frac{1}{2\pi i}\int_{\Gamma_{\varepsilon,\tau}^{\theta}\cup S_{\varepsilon}}e^{s t}\big(s-\omega_2(e^{-s\tau})\big)\hat{U}_{h}(s)\mathrm{d}s\\
      &~~~~~+\frac{1}{2\pi i}\int_{\Gamma_{\varepsilon,\tau}^{\theta}\cup S_{\varepsilon}}e^{s t}\omega_2(e^{-s\tau})\Big(\hat{U}_{h}(s)-\widehat{\tilde{U}_{h}}(s)\Big)\mathrm{d}s\\
      &~~~~~-\sum_{\substack{p=-\infty \\ p\ne 0}}^{+\infty}\frac{1}{2\pi i}\int_{\Gamma_{0,\tau}^{\theta}}e^{(s+i2\pi p/\tau) t}\omega_{2}(e^{-s\tau})
                       \widehat{\tilde{U}_{h}}(s+i2\pi p/\tau)\mathrm{d}s \\
                       &=I_{1}+I_{2}+I_{3}+I_{4},
    \end{split}
  \end{equation}
  where $\hat{U}_{h}(s)$ and $\widehat{\tilde{U}_{h}}(s)$ are given respectively by \eqref{eq:hatUh} and \eqref{eq:bUh} with $\hat{f}_h(s)\equiv 0$. By \eqref{eq:hUhs}, the estimation of the first item $I_{1}$ in \eqref{eq:gerror} is as follows
  \begin{equation*}
    \begin{split}
      \|I_{1}\|&\leq c \|u^0\|\int_{\Gamma_{\varepsilon}^{\theta}\backslash\Gamma_{\varepsilon,\tau}^{\theta}}
      |e^{st}||s|^{-1}|\mathrm{d}s|
         \leq c\|u^0\|\int_{\frac{\pi}{\tau\sin\theta}}^{+\infty}
      e^{\rho t\cos\theta}\rho^{-1}\mathrm{d}\rho\\
         &\leq c\tau^2\|u^0\|\int_{\frac{\pi}{\tau\sin\theta}}^{+\infty}
      e^{\rho t\cos\theta}\rho\mathrm{d}\rho
          \leq c \tau^2 t^{-2}\|u^0\|.
    \end{split}
  \end{equation*}
  Let $\varepsilon=t^{-1}$ when $t\ge\tau$ and $\varepsilon=\tau^{-1}$ when $0<t\le\tau$, then $\varepsilon\le t^{-1}$. It follows from Lemma \ref{lem:fbdf2err} and \eqref{eq:hUhs} that
  \begin{equation}
    \begin{split}
      \|I_{2}\|&\leq c\tau^2\|u^0\|\int_{\Gamma_{\varepsilon, \tau}^{\theta}\cup S_{\varepsilon}}|e^{st}||s||\mathrm{d}s| \\
        &\leq c\tau^2\|u^0\|\left(\int_{\varepsilon}^{\frac{\pi}{\tau\sin\theta}}e^{\rho t\cos\theta}\rho\mathrm{d}\rho+\int_{-\theta}^{\theta}e^{\varepsilon t\cos\xi}\varepsilon^{2}\mathrm{d}\xi\right) \\
        &\leq c \tau^2 t^{-2}\|u^0\|.
    \end{split}
  \end{equation}
  Since the following estimate
  \begin{equation}\label{eq:dcdresolv}
    \begin{split}
      &\|(s^{\alpha}+A_{h})^{-1}
      -\left(\omega(e^{-s\tau})^{\alpha}+A_{h}\right)^{-1}\|\\
      &\leq \|\left(\omega(e^{-s\tau})^{\alpha}+A_{h}\right)^{-1}\|
      \|\omega(e^{-s\tau})^{\alpha}-s^{\alpha}\| \|\left(s^{\alpha}+A_{h}\right)^{-1}\| \\
      &\leq c\tau^2|s|^2|\omega(e^{-s\tau})|^{-\alpha}
    \end{split}
  \end{equation}
  holds by using Lemmas \ref{lem:fCN} and \ref{lem:fCNerr} for $s$ enclosed by curves $\Gamma_{0,\tau}^{\theta}$, $\Im(s)=\pm \pi/\tau$ and $\Gamma$, it arrives at
  \begin{equation}\label{eq:resolv2}
    \begin{split}
      &\|(s^{\alpha}+A_{h})^{-1}s
       -\left(\omega(e^{-s\tau})^{\alpha}+A_{h}\right)^{-1}\omega_{2}(e^{-s\tau})\|\\
      &\le\|(s^{\alpha}+A_{h})^{-1}(s-\omega_{2}(e^{-s\tau}))\|+\|\big((s^{\alpha}+A_{h})^{-1}
       -\left(\omega(e^{-s\tau})^{\alpha}+A_{h}\right)^{-1}\big)\omega_{2}(e^{-s\tau})\|\\
      &\leq c\tau^{2}(|s|^{3-\alpha}+|s|^2|\omega(e^{-s\tau})|^{-\alpha}|\omega_{2}(e^{-s\tau})|),
    \end{split}
  \end{equation}
  for $s$ enclosed by curves $\Gamma_{0,\tau}^{\theta}$, $\Im(s)=\pm \pi/\tau$ and $\Gamma$.
  With \eqref{eq:hatUh}, \eqref{eq:bUh}, \eqref{eq:resolv2}, Lemmas \ref{lem:fbdf2err} and \ref{lem:fCNerr}, it yields the estimate of $I_{3}$ in \eqref{eq:gerror} as follows
  \begin{equation}
    \begin{split}
      \|I_{3}\|&\leq c\tau^2\|u^0\|\int_{\Gamma_{\varepsilon,\tau}^{\theta}\cup S_{\varepsilon}}|e^{s t}||\omega_2(e^{-s\tau})|\big(|s|^{3-\alpha}+|\omega(e^{-s\tau})|^{-\alpha}|\omega_2(e^{-s\tau})||s|^{2}\big)|s|^{\alpha-3}|\mathrm{d}s|  \\
          &\leq c\tau^2\|u^0\|\left(\int_{\varepsilon}^{\frac{\pi}{\tau\sin\theta}}e^{\rho t\cos\theta }\rho\mathrm{d}\rho+\int_{-\theta}^{\theta}e^{ \varepsilon t\cos\xi}\varepsilon^{2}\mathrm{d}\xi\right) \\
      &\leq c \tau^2 t^{-2}\|u^0\|.
    \end{split}
  \end{equation}
  In addition, from the inequality for any $\nu>-1$
  \begin{equation}
    \sum_{p=1}^{+\infty}p^{-\nu-2}\leq 1+\int_{1}^{+\infty}p^{-\nu-2}\mathrm{d}p\leq 1+\frac{1}{1+\nu},
  \end{equation}
  it follows that the fourth item $I_{4}$ in \eqref{eq:gerror} satisfies
  \begin{equation*}
    \begin{split}
      \|I_{4}\|&\leq c\|u^0\|\sum_{p=1}^{+\infty}\int_{\Gamma_{0,\tau}^{\theta}}|e^{st}||\omega_2(e^{-s\tau})|^{2}|\omega(e^{-s\tau})|^{-\alpha}|s+i2p\pi/\tau|^{\alpha-3}|\mathrm{d}s|
                  \\
                 &\leq c \tau^{2}\|u^0\|\sum_{p=1}^{+\infty}p^{\alpha-3}\int_{0}^{\frac{\pi}{\tau\sin\theta}}e^{ \rho t\cos\theta}\rho\mathrm{d}\rho   \\
                 &\leq c \tau^{2}t^{-2}\|u^0\|,
    \end{split}
  \end{equation*}
  where $\rho^{2-\alpha}\le\rho\tau^{\alpha-1}$ is applied as $0<\alpha\le 1$ and $\rho\in(0,\frac{\pi}{\tau\sin\theta})$.
  Therefore, the result \eqref{eq:81} is obtained.
\end{proof}

Note that the solution $\tilde{u}_{h}^n$ of \eqref{eq:CN} satisfies $\tilde{u}_{h}^n=\tilde{u}_{h}(t_n)$ with $\tilde{u}_{h}(t)$ given by \eqref{eq:tuh} or \eqref{eq:buhE}. Then the result in Lemma \ref{lem:herr} directly implies the error estimate of scheme \eqref{eq:CN} for solving \eqref{eq:tfdesub} of homogenous case. The result is stated in the following theorem.

\begin{theorem}\label{thm:herr}
  Assume $u^0(x)\in L^2(\Omega)$ and $f(x,t)\equiv0$. Let $u_{h}(t)$ and $\tilde{u}_{h}^n$ be the solutions of\eqref{eq:sFEO} and \eqref{eq:CN}, respectively. Then we have
  \begin{equation}\label{eq:thmherr}
    \|u_{h}(t_n)-\tilde{u}_{h}^n\|\leq c\tau^2 t_n^{-2}\|u^0\|.
  \end{equation}
\end{theorem}

Next we consider the error estimate for the inhomogeneous case. The Taylor expansion of $f_h(t)$ is of the form
\begin{equation}\label{eq:TeFh}
  f_h(t)=f_h(0)+tf_h'(0)+t*f_h''(t),
\end{equation}
where ``$*$" represents the convolution operation.
Thus, we need to obtain the following estimate for the case $t^{\mu}g_h(x)$ at first. The proof is analogous to that of Lemma \ref{lem:herr}.
\begin{lemma}\label{lem:inherr}
  Let $f_h(t)=t^{\mu}g_h(x)$ with $\mu\ge 0$, $E_{h}(\cdot)$ and $E_h^{\tau}(\cdot)$ be the operators given by \eqref{eq:Eh} and \eqref{eq:Eht}, respectively. Then we obtain
  \begin{equation}\label{eq:82}
    \|\left(E_{h}-E_h^{\tau}\right)*f_h(t)\|\leq c\tau^2 t^{\alpha+\mu-2}\|g_h\|,\quad\quad \forall~t\in(0,T].
  \end{equation}
\end{lemma}
\begin{proof}
  With $u^0(x)\equiv0$ and $\hat{f_h}(s)=\Gamma(1+\mu)s^{-\mu-1}g_h(x)$, it follows from \eqref{eq:hatUh}, \eqref{eq:bUh} and \eqref{eq:dcdresolv} that
\begin{equation*}
    \|\widehat{U}_{h}(s)-\widehat{\tilde{U}_{h}}(s)\|\le c\tau^2|\omega(e^{-s\tau})|^{-\alpha}|s|^{-\mu}\|g_h\|.
\end{equation*}
  Then we have from \eqref{eq:uh1}, \eqref{eq:hUhs}, Lemma \ref{lem:fCNerr} and Lemma \ref{lem:uhn} that
  \begin{align*}
      \|\big(E_{h}(t)-E_h^{\tau}(t)\big)*f_h(t)\|&\le c\int_{\Gamma_{\varepsilon}^{\theta}\backslash\Gamma_{\varepsilon,\tau}^{\theta}}
      |e^{st}|
      |s|\|\hat{U}_{h}(s)\||\mathrm{d}s|\\
      &~~~~~+c\int_{\Gamma_{\varepsilon,\tau}^{\theta}\cup S_{\varepsilon}}
      |e^{s t}||s-\omega_2(e^{-s\tau})|\|\hat{U}_{h}(s)\||\mathrm{d}s|\\
      &~~~~~+c\int_{\Gamma_{\varepsilon,\tau}^{\theta}\cup S_{\varepsilon}}
      |e^{st}||\omega_2(e^{-s\tau})|\|\hat{U}_{h}(s)-\widehat{\tilde{U}_{h}}(s)\||\mathrm{d}s|\\
      &~~~~~+\sum_{\substack{p=-\infty \\ p\ne 0}}^{+\infty}c\int_{\Gamma_{0,\tau}^{\theta}}|e^{(s+i2\pi p/\tau) t}||\omega_{2}(e^{-s\tau})|\|\widehat{\tilde{U}_{h}}(s+i2\pi p/\tau)\||\mathrm{d}s| \\
      &\le c\tau^2\|g_h\|\int_{\frac{\pi}{\tau\sin\theta}}^{+\infty}
      e^{\cos\theta \rho t}\rho^{1-\alpha-\mu}\mathrm{d}\rho\\
      &~~~~~+c\tau^2\|g_h\|\left(\int_{\varepsilon}^{\frac{\pi}{\tau\sin\theta}}e^{\rho \cos\theta t}\rho^{1-\alpha-\mu}\mathrm{d}\rho+\int_{-\theta}^{\theta}e^{ \varepsilon t\cos\xi}\varepsilon^{2-\alpha-\mu}\mathrm{d}\xi\right)\\
      &~~~~~+c\tau^{2}\|g_h\|\sum_{p=1}^{+\infty}p^{-\mu-2}\int_{0}^{\frac{\pi}{\tau\sin\theta}}e^{ \rho t\cos\theta}\rho^{1-\alpha-\mu}\mathrm{d}\rho\\
      &\le c\tau^2t^{\alpha+\mu-2}\|g_h\|,
  \end{align*}
  where $\rho^{1-\alpha}\le c\rho^{1-\alpha-\mu}\tau^{-\mu}$ is applied as $\mu\ge0$ and $\rho\in(0,\frac{\pi}{\tau\sin\theta})$. The result \eqref{eq:82} is obtained.
\end{proof}

\begin{theorem}\label{thm:errCN}
  Assume $u^0(x)\in L^2(\Omega)$, $f\in W^{1,\infty}(0,T;L^2(\Omega))$ and $\int_{0}^t(t-\zeta)^{\alpha-1}\|f''(\zeta)\|\mathrm{d}\zeta<\infty$. Let $u_{h}$ and $\tilde{u}_{h}^n$ be the solutions of \eqref{eq:sFEO} and the scheme \eqref{eq:CN}, respectively. Then for $1\le n\le N$, it holds that
  \begin{equation}\label{eq:thmerrCN}
    \|u_{h}(t_{n})-\tilde{u}_{h}^{n}\|\leq c\tau^2 \Big(t_{n}^{-2}\|u^0\| +t_{n}^{\alpha-2}\|f(0)\|+t_{n}^{\alpha-1}\|f'(0)\|+\int_{0}^{t_n}(t_n-\zeta)^{\alpha-1}\|f''(\zeta)\|\mathrm{d}\zeta\Big).
  \end{equation}
\end{theorem}
\begin{proof}
 It follows from \eqref{eq:uhE}, \eqref{eq:buhE} and \eqref{eq:TeFh} that
  \begin{equation}\label{eq:errCN1}
    \begin{aligned}
      \|u_{h}(t_{n})-\tilde{u}_{h}^{n}\|&\leq \|\left(\bar{E}_h(t_n)-\bar{E}_h^{\tau}(t_n)\right)P_hu^0\|
      +\|\big((E_{h}(t)-E_h^{\tau}(t))*1\big)(t_n)f_h(0)\|\\
      &~~~~~+\|\big((E_{h}(t)-E_h^{\tau}(t))*t\big)(t_n)f_h'(0)\|\\
      &~~~~~+\|\big((E_{h}(t)-E_h^{\tau}(t))*t*f_h''(t)\big)(t_n)\|.
    \end{aligned}
  \end{equation}
  Then \eqref{eq:thmerrCN} can be derived from \eqref{eq:errCN1} by Lemmas \ref{lem:herr} and \ref{lem:inherr}.
\end{proof}

In the following, the error estimates of scheme \eqref{eq:CN} for the source terms in the forms of $f(x,t)=t^{\mu}*g(x,t)$ with $\mu>-1$ and $f(x,t)=t^{\mu}g(x,t)$ with $\mu>0$ are established in Theorems \ref{thm:errCNcst} and \ref{thm:errCNdst}, respectively. The proofs are based on Lemmas \ref{lem:herr} and \ref{lem:inherr}.

\begin{theorem}[Source terms $f=t^{\mu}*g(t)$]\label{thm:errCNcst}
  Assume that $u^0(x)\in L^2(\Omega)$, and $f(t)=t^{\mu}*g(t)$ with $\mu>-1$ satisfying $g\in W^{1,\infty}(0,T;L^2(\Omega))$ and $\int_{0}^t(t-\zeta)^{\min(\alpha+\mu,0)}\|g''(\zeta)\|\mathrm{d}\zeta<\infty$. Let $u_{h}$ and $\tilde{u}_{h}^n$ be the solutions of \eqref{eq:sFEO} and the scheme \eqref{eq:CN}, respectively. Then for $1\le n\le N$, it holds that
  \begin{equation}\label{eq:thmerrCNcst}
    \|u_{h}(t_{n})-\tilde{u}_{h}^{n}\|\leq c\tau^2 \Big(t_{n}^{-2}\|u^0\| +t_{n}^{\alpha+\mu-1}\|g(0)\|+t_{n}^{\alpha+\mu}\|g'(0)\|+\int_{0}^{t_n}(t_n-\zeta)^{\alpha+\mu}\|g''(\zeta)\|\big)\mathrm{d}\zeta\Big).
  \end{equation}
\end{theorem}
\begin{proof}
  In view of $g\in W^{1,\infty}(0,T;L^2(\Omega))$ and $g''\in L^{1}(0,T;L^2(\Omega))$, we have the Taylor expansion of $g(t)$ as follows
  \begin{equation}\label{eq:TEgt}
    g(t)=g(0)+tg'(0)+t*g''(t).
  \end{equation}
  By the identity $t^{\mu}*t=\frac{t^{\mu+2}}{(\mu+1)(\mu+2)}$,  $f_h(t)=t^{\mu}*g_h(t)$ can be reformulated as
  \begin{equation}
   f_h(t)=\frac{t^{\mu+1}}{\mu+1}g_h(0)+\frac{t^{\mu+2}}{(\mu+1)(\mu+2)}g_h'(0)
    +\frac{1}{(\mu+1)(\mu+2)}t^{\mu+2}*g_h''(t).
  \end{equation}
  Then together with Lemmas \ref{lem:herr} and \ref{lem:inherr}, the result \eqref{eq:thmerrCNcst} can be derived by the similar argument as in \eqref{eq:errCN1}.
\end{proof}

\begin{theorem}[Source terms $f=t^{\mu}g(t)$]\label{thm:errCNdst}
  Assume that $u^0(x)\in L^2(\Omega)$, and $f(x,t)=t^{\mu}g(x,t)$ with $\mu\ge0$ satisfy $g\in W^{1,\infty}(0,T;L^2(\Omega))$, $\int_{0}^{t}\|g''(\zeta)\|\mathrm{d}\zeta<\infty$ and $\int_{0}^t(t-\zeta)^{\alpha-1}\zeta^{\mu}\|g''(\zeta)\|\mathrm{d}\zeta<\infty$. Let $u_{h}$ and $\tilde{u}_{h}^n$ be the solutions of \eqref{eq:sFEO} and \eqref{eq:CN}, respectively. Then for $1\le n\le N$, it holds that
  \begin{equation}\label{eq:errCNdst}
    \begin{aligned}
      \|u_{h}(t_{n})-\tilde{u}_{h}^{n}\|\leq c\tau^2 \Big(&t_{n}^{-2}\|u^0\| +t_{n}^{\alpha+\mu-2}\|g(0)\|+t_{n}^{\alpha+\mu-1}\|g'(0)\|\\
      &+t_n^{\alpha+\mu-1}\int_{0}^{t_n}\|g''(\zeta)\|\mathrm{d}\zeta+\int_{0}^{t_n}(t_n-\zeta)^{\alpha-1}\zeta^{\mu}\|g''(\zeta)\|\mathrm{d}\zeta\Big).
    \end{aligned}
  \end{equation}
\end{theorem}
\begin{proof}
  For the case $\mu=0$, the result \eqref{eq:errCNdst} can be directly derived from Theorem \ref{thm:errCN}. We next consider the case $\mu>0$.
  The Taylor expansion of $g$ in \eqref{eq:TEgt} yields that
  \begin{equation}\label{eq:errCNdst1}
    f_h(t)=t^{\mu}g_h(0)+t^{\mu+1}g_h'(0)
    +t^{\mu}(t*g_h''(t)), 
  \end{equation}
  where $q_h(t)=t^{\mu}(t*g_h''(t))$, $q_h(0)=0$ and $q_h'(t)=\mu t^{\mu-1}(t*g_h''(t))+t^{\mu}(1*g_h''(t))$ by the argument in \cite{ChenSZ:2022x}. Furthermore, $q_h'(0)=0$ since $\|q_h'(t)\|\le (\mu+1)t^{\mu}(1*\|g_h''(t)\|)$, and
  $$
    q_h''(t)=\mu(\mu-1) t^{\mu-2}(t*g_h''(t))+2\mu t^{\mu-1}(1*g_h''(t))+t^{\mu}g_h''(t),
  $$
  which satisfies
  \begin{equation}\label{eq:errCNdst3}
    \|q_h''(t)\|\le c\big(t^{\mu-1}(1*\|g_h''(\zeta)\|)+t^{\mu}\|g_h''(t)\|\big).
  \end{equation}
  Then it implies from Lemma \ref{lem:inherr} that
  \begin{equation}\label{eq:errCNdst2}
    \begin{aligned}
      &\|\big((E_{h}(t)-E_h^{\tau}(t))*q_h(t)\big)(t_n)\|\\
      &=\|\big((E_{h}(t)-E_h^{\tau}(t))*t*q_h''(t)\big)(t_n)\|\\
      &\le c\tau^2\int_{0}^{t_n}(t_n-\zeta)^{\alpha-1}\|q_h''(\zeta)\|\mathrm{d}\zeta\\
      &\le c\tau^2\Big(\int_0^{t_n}(t_n-\zeta)^{\alpha-1}\zeta^{\mu-1}\mathrm{d}\zeta\int_{0}^{t_n}\|g''(z)\|\mathrm{d}z+\int_{0}^{t_n}(t_n-\zeta)^{\alpha-1}\zeta^{\mu}\|g''(\zeta)\|\mathrm{d}\zeta\Big)\\
      &\le c\tau^2\Big(t_n^{\alpha+\mu-1}\int_{0}^{t_n}\|g''(\zeta)\|\mathrm{d}\zeta+\int_{0}^{t_n}(t_n-\zeta)^{\alpha-1}\zeta^{\mu}\|g''(\zeta)\|\mathrm{d}\zeta\Big).
    \end{aligned}
  \end{equation}
  Then the result \eqref{eq:errCNdst} is derived from \eqref{eq:errCNdst1}, \eqref{eq:errCNdst3}, \eqref{eq:errCNdst2}, Lemmas \ref{lem:herr} and \ref{lem:inherr}.
\end{proof}

\begin{remark}
  The result in Theorem \ref{thm:errCN} reveals that the Crank-Nicolson scheme \eqref{eq:CN} achieves second-order accuracy for problem \eqref{eq:tfdesub} with certain smooth source terms, which is consistent with the results of Crank-Nicolson schemes in \cite[Theorem 3.8 and Theorem 3.13]{JinLZ:2018a} and \cite[Theorem 2]{WangWY:2021}. However, those error estimates in \cite{JinLZ:2018a,WangWY:2021} are invalid for some source terms with lower regularity such as $f(x,t)=t^{\mu}g(x)$ with $\mu\in(0,1)$. For $u^0(x)\equiv0$ and $f(x,t)=t^{\mu}g(x)$ with $\mu\ge0$, the result in Lemma \ref{lem:inherr} or Theorem \ref{thm:errCNdst} shows the optimal second-order error estimate as follows
  \begin{equation}
    \|u_{h}(t_{n})-\tilde{u}_{h}^{n}\|\leq ct_{n}^{\alpha+\mu-2}\tau^{2}.
  \end{equation}
\end{remark}

\begin{remark}
  For singular source terms $f(x,t)=t^{\mu}g(x)$ with $\mu\in(-1,0)$, the Crank-Nicolson scheme \eqref{eq:CN} can not preserve the optimal second-order accuracy. By the similar approach for Lemma \ref{lem:inherr}, we can obtain that the error estimate of the scheme \eqref{eq:CN} is
  \begin{equation}
    \|u_{h}(t_{n})-\tilde{u}_{h}^{n}\|\leq ct_{n}^{\alpha-2}\tau^{2+\mu}
  \end{equation}
  for the case $u^0(x)\equiv0$ and $f(x,t)=t^{\mu}g(x)$ with $\mu\in(-1,0)$.
\end{remark}

\section{Crank-Nicolson scheme for singular source terms}\label{sec:4}
In this section, we consider designing an alternative Crank-Nicolson scheme with the optimal second-order accuracy for solving \eqref{eq:tfdesub} with singular source terms, such as $f(x,t)=t^{\mu}g(x,t)$ with $\mu\in(-1,0)$. To recover the optimal second-order rate of convergence of Crank-Nicolson method, we need to introduce a function $\tilde{F}_h(t)$ satisfying
\begin{equation}\label{eq:f2}
  D_t\tilde{F}_h(t)=F_h(t),~~\tilde{F}_h(0)=0
\end{equation}
with $F_h(t)$ given by \eqref{eq:f1}. Replacing $F_h(t)$ in \eqref{eq:daU} by $D_t\tilde{F}_h(t)$, then we obtain an equivalent form of \eqref{eq:daU} as follows
\begin{equation}\label{eq:daU2}
  {^{C\!}}D^{\alpha}_tU_{h}+A_{h}U_{h}=D_t\tilde{F}_h+D_t\frac{t^{2-\alpha}}{\Gamma(3-\alpha)}u_{h}(0),\quad
  U_{h}(0)=0.
\end{equation}
Next we will design a new Crank-Nicolson scheme used for numerically solving \eqref{eq:daU2} and \eqref{eq:u1}.

\subsection{Crank-Nicolson scheme II}
We denote $\tilde{U}_{h}(t)$ and $\tilde{u}_{h}(t)$ as approximations to $U_{h}(t)$ and $u_{h}(t)$ solving \eqref{eq:daU2} and \eqref{eq:u1}, which satisfy the difference equations
\begin{align}
  &D_{\tau}^{\alpha}\tilde{U}_{h}(t)+(1-\frac{\alpha}{2})A_{h}\tilde{U}_{h}(t)
    +\frac{\alpha}{2}A_{h}\tilde{U}_{h}(t-\tau)=(1-\frac{\alpha}{2})D_{\tau}\tilde{F}_{h}(t)
    +\frac{\alpha}{2}D_{\tau}\tilde{F}_{h}(t-\tau)\notag\\
  &\hskip4.6cm+(1-\frac{\alpha}{2})D_{\tau}\frac{t^{2-\alpha}}{\Gamma(3-\alpha)}u_{h}(0)
    +\frac{\alpha}{2}D_{\tau}\frac{(t-\tau)^{2-\alpha}}{\Gamma(3-\alpha)}u_{h}(0),\label{eq:nDECN}\\
  &\tilde{u}_{h}(t)=D_{\tau}\tilde{U}_{h}(t)=\frac{1}{\tau}\Big(\frac{3}{2}
    \tilde{U}_{h}(t)-2\tilde{U}_{h}(t-\tau)+\frac{1}{2}\tilde{U}_{h}(t-2\tau)\Big)\label{eq:ncuCN}
\end{align}
for $t>0$, and prescribe $\tilde{U}_{h}(t)=0$, $\tilde{u}_{h}(t)=0$ for $t\le0$,
where $D_{\tau}$ denotes the second-order BDF operator and $D_{\tau}^{\alpha}$ the Gr\"unwald-Letnikov difference operator \eqref{eq:cCN}.

Taking $t=t_n=n\tau,~n=1,2,\cdots,N$ with $\tau=T/N$ in \eqref{eq:nDECN} and \eqref{eq:ncuCN}, we establish a new fully discrete Crank-Nicolson scheme of the form
\begin{equation}\label{eq:nCN}
  \left\{
  \begin{aligned}
    &\tau^{-\alpha}\sum_{j=0}^{n}\sigma_{j}\tilde{U}_{h}^{n-j}+(1-\frac{\alpha}{2})A_{h}\tilde{U}_{h}^{n}
    +\frac{\alpha}{2}A_{h}\tilde{U}_{h}^{n-1}
    =(1-\frac{\alpha}{2})D_{\tau}\tilde{F}_{h}^{n}
    +\frac{\alpha}{2}D_{\tau}\tilde{F}_{h}^{n-1}\\
    &\hskip4.0cm+(1-\frac{\alpha}{2})D_{\tau}\frac{t_{n}^{2-\alpha}}{\Gamma(3-\alpha)}u_{h}(0)
     +\frac{\alpha}{2}D_{\tau}\frac{t_{n-1}^{2-\alpha}}{\Gamma(3-\alpha)}u_{h}(0),\\
    &\tilde{u}_{h}^{n}=\tau^{-1}\Big(\frac{3}{2}\tilde{U}_{h}^{n}-2\tilde{U}_{h}^{n-1}
    +\frac{1}{2}\tilde{U}_{h}^{n-2}\Big)
  \end{aligned}\right.\tag{CN-II}
\end{equation}
with $\tilde{U}_{h}^{0}=0$, where $\tilde{u}_{h}^{n}:=\tilde{u}_{h}(t_{n})$, $\tilde{U}_{h}^{n}:=\tilde{U}_{h}(t_{n})$
and $\tilde{F}_{h}^{n}:=\tilde{F}_{h}(t_{n})$ with $\tilde{F}_h(\cdot)$ satisfying \eqref{eq:f2}.

Taking the Laplace transform on \eqref{eq:nDECN}, we obtain
\begin{equation}\label{eq:nbUh}
  \widehat{\tilde{U}_{h}}(s)=\omega_2(e^{-s\tau})\big(\omega(e^{-s\tau})^{\alpha}+A_{h}\big)^{-1}
  \big(s^{-2}\hat{f_{h}}(s)+s^{\alpha-3}u_{h}(0)\big).
\end{equation}
This implies from \eqref{eq:ncuCN} that
\begin{equation}\label{eq:nuhE}
  \tilde{u}_{h}(t)=\frac{1}{2\pi i}\int_{\Gamma}e^{st}\omega_2(e^{-s\tau})\widehat{\tilde{U}_{h}}(s)\mathrm{d}s.
\end{equation}
An alternative expression of $\tilde{u}_{h}(t)$ is
\begin{equation}\label{eq:ntuhE}
  \tilde{u}_{h}(t)
  =\bar{E}_h^{\tau}(t)u_h(0)+\int_0^t\tilde{E}_h^{\tau}(t-\zeta)f_h(\zeta)\mathrm{d}\zeta,
\end{equation}
where $\omega(\cdot)$ and $\omega_2(\cdot)$ are given by \eqref{eq:gfCN} and \eqref{eq:bdf2}, the operator $\bar{E}_h^{\tau}(t)$ is given by \eqref{eq:bEht}, and $\tilde{E}_h^{\tau}(\cdot)$ is an operator on $X_{h}$ given by
\begin{align}
  \tilde{E}_h^{\tau}(t)&=\frac{1}{2\pi i}\int_{\Gamma}e^{st}\omega_2(e^{-s\tau})^2(\omega(e^{-s\tau})^{\alpha}+A_{h})^{-1}s^{-2}\mathrm{d}s.\label{eq:tEht}
\end{align}

\subsection{Error estimates for singular source terms}\label{sec:4.2}
In this subsection, we first consider the error estimates for the singular source terms satisfying the following conditions in Assumption \ref{asm:af} as discussed in \cite{ZhouT:2022}, and then extend the result to singular source terms in the form of $f(x,t)=t^{\mu}g(x,t)$ with $\mu>-1$.
\begin{assumption}\label{asm:af}
  The singular source term $f(x,t)$ in \eqref{eq:tfdesub} is assumed to be in $L^1(0,T;L^{2}(\Omega))$ such that its Laplace transform with respect to time $t$ is analytic within the domain $\Sigma_{\theta}^{\varepsilon}=\big\{z\in\mathbb{C}:|z|\ge\varepsilon~\text{and}~|\mathrm{arg}z|<\theta\big\}$ for $\theta\in (\pi/2,\pi)$ and small $\varepsilon>0$, and $\|\hat{f}(s)\|\le c|s|^{-\mu-1}$ holds for some $\mu>-1$.
\end{assumption}

As mentioned in \cite{ZhouT:2022}, the singular source term $f(x,t)=t^{\mu}g(x)$ with $-1<\mu<0$ and $g(x)\in L^2(\Omega)$ satisfies the conditions in Assumption \ref{asm:af}.

\begin{remark}\label{rem:herr}
  For the homogenous problem \eqref{eq:tfdesub} with $f(x,t)\equiv0$, the Crank-Nicolson schemes \eqref{eq:CN} and \eqref{eq:nCN} are the same. Thus the error estimate of the scheme \eqref{eq:nCN} for the homogenous case is the same as Theorem \ref{thm:herr}.
\end{remark}

The error estimate of the new Crank-Nicolson scheme \eqref{eq:nCN} for the inhomogeneous case is established in the following theorem, which restores the optimal second-order accuracy for singular source terms satisfying Assumption \ref{asm:af}.

\begin{theorem}\label{thm:errnCN}
  Let $u^0(x)\equiv0$ and $f(x,t)$ in \eqref{eq:tfdesub} satisfy Assumption \ref{asm:af} (with $\mu>-1$). If $u_{h}$ and $\tilde{u}_{h}^n$ are solutions to \eqref{eq:sFEO} and \eqref{eq:nCN}, respectively, then it holds that
  \begin{equation}
    \|u_{h}(t_{n})-\tilde{u}_{h}^{n}\|\leq ct_n^{\alpha+\mu-2}\tau^2,\quad\quad n=1,2,\cdots, N.
  \end{equation}
\end{theorem}
\begin{proof}
  By the similar argument in Lemma \ref{lem:uhn}, it follows from \eqref{eq:uhE} and \eqref{eq:nuhE} that
  \begin{equation}\label{eq:errnCN1}
    \begin{split}
      u_{h}(t_n)-\tilde{u}_h^n
      &=\frac{1}{2\pi i}\int_{\Gamma_{\varepsilon}^{\theta}\backslash\Gamma_{\varepsilon,\tau}^{\theta}}e^{st_{n}}
      s\hat{U}_{h}(s)\mathrm{d}s\\
      &~~~~~+\frac{1}{2\pi i}\int_{\Gamma_{\varepsilon,\tau}^{\theta}\cup S_{\varepsilon}}e^{s t_{n}}\left(s\hat{U}_{h}(s)-\omega_{2}(e^{-s\tau})\widehat{\tilde{U}_{h}}(s)\right)\mathrm{d}s\\
      &~~~~~-\sum_{\substack{p=-\infty \\ p\ne 0}}^{+\infty}\frac{1}{2\pi i}\int_{\Gamma_{0,\tau}^{\theta}}e^{st_{n} }\omega_{2}(e^{-s\tau})\widehat{\tilde{U}_{h}}(s+i2\pi p/\tau)\mathrm{d}s \\
      &:=II_1+II_2+II_3.
    \end{split}
  \end{equation}
  With \eqref{eq:nbUh} and Assumption \ref{asm:af} on $f(t)$, it yields
  \begin{equation}\label{eq:nbUhs}
    \|\widehat{\tilde{U}_{h}}(s)\|\le c|\omega_2(e^{-s\tau})||\omega(e^{-s\tau})|^{-\alpha}|s^{-\mu-3}.
  \end{equation}
  Then we have from \eqref{eq:nbUhs}, Lemmas \ref{lem:fbdf2err} and \ref{lem:fCNerr} that
  \begin{equation*}
    \|II_1\|\le c\tau^2\int_{\frac{\pi}{\tau\sin\theta}}^{\infty}e^{\rho t_n\cos\theta}\rho^{-\alpha-\mu+1}\mathrm{d}\rho\le ct_n^{\alpha+\mu-2}\tau^2.
  \end{equation*}
  From \eqref{eq:hUhs}, \eqref{eq:nbUh}, \eqref{eq:nbUhs}, \eqref{eq:resolv2} and Lemma \ref{lem:fbdf2err}, it follows that
  \begin{equation*}
    \begin{aligned}
      &\|s\hat{U}_{h}(s)-\omega_{2}(e^{-s\tau})\widehat{\tilde{U}_{h}}(s)\|\\
      &\le\|\big(s-\omega_2(e^{-s\tau})\big)\hat{U}_{h}(s)\|
       +\|\omega_2(e^{-s\tau})\big(\hat{U}_{h}(s)-\widehat{\tilde{U}_{h}}(s)\big)\|\\
      &\le \|\big(s-\omega_2(e^{-s\tau})\big)\hat{U}_{h}(s)\|
       +|\omega_{2}(e^{-s\tau})|\|(s^{\alpha}+A_{h})^{-1}s
        -\left(\omega(e^{-s\tau})^{\alpha}+A_{h}\right)^{-1}\omega_{2}(e^{-s\tau})\|\|\hat{\tilde{F}}_{h}(s)\|\\
      &\le c\tau^{2}\big(|s|^{1-\alpha-\mu}+|\omega_{2}(e^{-s\tau})|(|s|^{3-\alpha}+|s|^2||\omega(e^{-s\tau})|^{-\alpha}|\omega_{2}(e^{-s\tau})|)|s|^{-\mu-3}\big)\\
      &\le c\tau^2|s|^{1-\alpha-\mu}.
    \end{aligned}
  \end{equation*}
  This gives
  \begin{equation}
    \|II_2\|\le c\tau^2\int_{\Gamma_{\varepsilon,\tau}^{\theta}\cup S_{\varepsilon}}|e^{s t_{n}}||s|^{1-\alpha-\mu}|\mathrm{d}s|
    \le ct_n^{\alpha+\mu-2}\tau^2.
  \end{equation}
  Furthermore, by Lemmas \ref{lem:fbdf2err}, \ref{lem:fCN} and \ref{lem:fCNerr}, we get
  \begin{equation}
    \begin{aligned}
      \|II_3\|&\le c\sum_{p=1}^{\infty}\int_{\Gamma_{0,\tau}^{\theta}}|e^{st_{n}| }|s|^{2-\alpha}|s+i2\pi p/\tau|^{-\mu-3}|\mathrm{d}s|\\
      &\le c\tau^{3+\mu}\sum_{p=1}^{+\infty}p^{-\mu-3}\int_{0}^{\frac{\pi}{\tau\sin\theta}}e^{ \rho t_n\cos\theta}\rho^{2-\alpha}\mathrm{d}\rho\\
      &\le ct_n^{\alpha+\mu-2}\tau^2,
    \end{aligned}
  \end{equation}
  where $\rho^{2-\alpha}\le c\rho^{1-\alpha-\mu}\tau^{-\mu-1}$ is applied as $\mu>-1$ and $\rho\in(0,\frac{\pi}{\tau\sin\theta})$. This completes the proof.
\end{proof}

Using similar analysis techniques for Lemma \ref{lem:inherr} and Theorem \ref{thm:errnCN}, we can easily derive the following error bound.

\begin{lemma}\label{lem:inherr2}
  Let $f_h(x,t)=t^{\mu}g_h(x)$ with $\mu>-1$, $E_{h}(\cdot)$ and $\tilde{E}_h^{\tau}(\cdot)$ be given by \eqref{eq:Eh} and \eqref{eq:tEht}, respectively. Then we have
  \begin{equation}
    \|\big(E_{h}-\tilde{E}_h^{\tau}\big)*f_h(t)\|\leq c\tau^2 t^{\alpha+\mu-2}\|g\|,\quad\quad \forall~ t\in(0,T].
  \end{equation}
\end{lemma}

Next, we establish the error estimate of scheme \eqref{eq:nCN} for singular source terms $f(x,t)=t^{\mu}g(x,t)$ with $-1<\mu<0$, where $g(t)$ has the Taylor expansion $g(t)=g(0)+tg'(0)+t*g''(t)$. The result is stated in the following theorem, the proof of which is based on Lemma \ref{lem:inherr2} and some techniques used in \cite[Lemma 5.5]{ChenSZ:2022x}.
\begin{theorem}\label{thm:errnCN2}
  Let $u^0(x)\equiv0$ and $f(x,t)=t^{\mu}g(x,t)$ with $-1<\mu<0$, where $g\in W^{1,\infty}(0,T;L^2(\Omega))$, $\int_{0}^t(t-\zeta)^{\alpha-1}\zeta^{\mu}\|g''(\zeta)\|\mathrm{d}\zeta<\infty$ and $\int_0^{t}\zeta^{\frac{\mu-1}{2}}\|g''(\zeta)\|\mathrm{d}\zeta<\infty$. If $u_{h}$ and $\tilde{u}_{h}^n$ are the solutions to \eqref{eq:sFEO} and \eqref{eq:nCN}, respectively, then we have that the error bound
  \begin{equation}\label{eq:thmerrnCN2}
    \begin{aligned}
      \|u_{h}(t_{n})-\tilde{u}_{h}^{n}\|\leq c\tau^2\Big(&t_n^{\alpha+\mu-2}\|g(0)\|+t_n^{\alpha+\mu-1}\|g'(0)\|\\
      &+\int_0^{t_n}(t_n-\zeta)^{\alpha-1}\zeta^{\mu}\|g''(\zeta)\|\mathrm{d}\zeta+t_n^{\alpha+\frac{\mu-1}{2}}\int_0^{t_n}\zeta^{\frac{\mu-1}{2}}\|g''(\zeta)\|\mathrm{d}\zeta\Big)
    \end{aligned}
  \end{equation}
  holds for $1\le n\le N$.
\end{theorem}
\begin{proof}
  It follows from $g(t)=g(0)+tg'(0)+t*g''(t)$ that
  \begin{equation}\label{eq:tFht}
    f_h(t)=t^{\mu}g_h(t)
    =t^{\mu}g_h(0)+t^{\mu+1}g_h'(0)+t^{\mu}(t*g_h''(t)).
  \end{equation}
  The error estimates for the first two terms in the right hand side of \eqref{eq:tFht} can be directly obtained from Theorem \ref{thm:errnCN}. For the third term, let $q_h(t)=t^{\mu}(t*g_h''(t))$ with $q_h(0)=0$, then it has
  \begin{align}
  q_h(t)&=tq_h'(0)+t*q_h''(t),\label{eq:errnCN24}\\
    q_h'(t)&=\mu t^{\mu-1}(t*g_h''(t))+t^{\mu}(1*g_h''(t)),\\
    q_h''(t)&=\mu(\mu-1)t^{\mu-2}(t*g_h''(t))+2\mu t^{\mu-1}(1*g_h''(t))+t^{\mu}g_h''(t),\label{eq:errnCN21}
  \end{align}
  which implies that
  \begin{equation}\label{eq:errnCN22}
    \|q_h'(t)\|\le(\mu+1)t^{\mu}\int_0^t\|g_h''(\zeta)\|\mathrm{d}\zeta\le (\mu+1)\int_0^t\zeta^{\mu}\|g_h''(\zeta)\|\mathrm{d}\zeta,\quad -1<\mu<0.
  \end{equation}
  Then $q_h'(0)=0$. We derive from \eqref{eq:errnCN24} and Lemma \ref{lem:inherr2} that
  \begin{equation*}\label{eq:errnCN23}
    \|\big(\tilde{E}_{h}(t)-\tilde{E}_h^{\tau}(t)\big)*q_h(t))(t_n)\|
    =\|\big(\tilde{E}_{h}(t)-\tilde{E}_h^{\tau}(t)\big)*t*q_h''(t))(t_n)\|\le c\tau^2\int_0^{t_n}(t_n-\zeta)^{\alpha-1}\|q_h''(\zeta)\|\mathrm{d}\zeta.
  \end{equation*}
  In addition, it follows that
  \begin{align*}
    \int_0^{t_n}(t_n-\zeta)^{\alpha-1}\|\zeta^{\mu-1}(1*g_h''(\zeta))\|\mathrm{d}\zeta
    &\le c\int_0^{t_n}(t_n-\zeta)^{\alpha-1}\zeta^{\frac{\mu-1}{2}}\int_0^{\zeta}z^{\frac{\mu-1}{2}}\|g_h''(z)\|\mathrm{d}z\mathrm{d}\zeta\\
    &\le ct_n^{\alpha+\frac{\mu-1}{2}}\int_0^{t_n}z^{\frac{\mu-1}{2}}\|g''(z)\|\mathrm{d}z
  \end{align*}
  and
  \begin{equation*}
   t^{\alpha-1}\int_0^{t}\zeta^{\mu}\|g_h''(\zeta)\|\mathrm{d}\zeta
   \le \int_0^{t}(t-\zeta)^{\alpha-1}\zeta^{\mu}\|g''(\zeta)\|\mathrm{d}\zeta.
  \end{equation*}
  The similar estimate holds for the term $t^{\mu-2}(t*g_h''(t))$ in \eqref{eq:errnCN21}.
  Thus, the result \eqref{eq:thmerrnCN2} is obtained from Lemma \ref{lem:inherr2} and the above discussions.
\end{proof}

\section{Numerical examples}\label{sec:5}
In this section, we present some numerical examples to verify the theoretical convergence results of the proposed Crank-Nicolson schemes \eqref{eq:CN} and \eqref{eq:nCN} for solving the sub-diffusion problem \eqref{eq:tfdesub} with both nonsingular and singular source terms.
Since the exact solutions in the following numerical examples are unknown, we utilize the formula
$
  \log_2(\|e_h^{N}\|/\|e_h^{2N}\|)
$
to verify the convergence rates of the schemes, where $e_h^{N}:=\tilde{u}_{h}^N-\tilde{u}_{h}^{N/2}$ and $\tilde{u}_{h}^N$ refers to the numerical solutions at time $T$ by the fully discrete schemes with the time step size $\tau=T/N$ and spatial mesh size $h$.
For one dimensional case, the spatial interval $\Omega=(0,1)$ is equally divided into subintervals with a mesh size $h=1/128$ for the finite element discretization. The domain $\Omega=(0,1)^2$ in two dimensional case is uniformly partitioned into triangles with the mesh size $h=1/128$.

\begin{example}\label{exm:f0}
  Consider the sub-diffusion problem \eqref{eq:tfdesub} with $T=1$, initial value $u^0(x)\equiv0$ and the following data:
  \begin{itemize}
    \item[\rm (a)] $\Omega=(0,1)$ and $f(x,t)=(1+t^{\mu})x^{-\frac{1}{4}}$ with $0<\mu<1$;
    \item[\rm (b)] $\Omega=(0,1)^2$ and $f(x,t)=(1+t^{\mu})\chi_{[\frac{1}{4},\frac{3}{4}]\times[\frac{1}{4},\frac{3}{4}]}(x)$, where $0<\mu<1$ and $\chi_{[\frac{1}{4},\frac{3}{4}]\times[\frac{1}{4},\frac{3}{4}]}(x)$ is the indicator function over $[\frac{1}{4},\frac{3}{4}]\times[\frac{1}{4},\frac{3}{4}]$.
  \end{itemize}
\end{example}

\begin{table}[htb!]
  \centering
  \caption{Errors and convergence rates by schemes CN-JLZ \cite{JinLZ:2018a}, CN-WWY \cite{WangWY:2021} and scheme \eqref{eq:CN} for case (a) of Example \ref{exm:f0}.}\label{tab:f0a}
  \setlength{\tabcolsep}{8.0pt}\small
  \begin{tabular}{lccccccc}
    \toprule
    scheme & $\alpha$ & $\mu$ & $N=$80 & 160 & 320 & 640 & rate  \\
    \midrule
    CN-JLZ & 0.2 & 0.1 & 6.1084E-06 & 2.9368E-06 & 1.3913E-06 & 6.5432E-07 & 1.07  \\
           &     & 0.5 & 5.0993E-07 & 2.2208E-07 & 8.8577E-08 & 3.3787E-08 & 1.31  \\
           & 0.5 & 0.1 & 1.1141E-05 & 5.4485E-06 & 2.6029E-06 & 1.2294E-06 & 1.06  \\
           &     & 0.5 & 6.1043E-07 & 3.3293E-07 & 1.4585E-07 & 5.8467E-08 & 1.13  \\
           & 0.8 & 0.1 & 8.5079E-06 & 4.3208E-06 & 2.1011E-06 & 1.0013E-06 & 1.03  \\
           &     & 0.5 & 1.2196E-07 & 1.2085E-07 & 8.2079E-08 & 3.8663E-08 & 0.55  \\
    \midrule
    CN-WWY & 0.2 & 0.1 & 6.1591E-06 & 2.9489E-06 & 1.3942E-06 & 6.5505E-07 & 1.08  \\
           &     & 0.5 & 5.6062E-07 & 2.3419E-07 & 9.1538E-08 & 3.4519E-08 & 1.34  \\
           & 0.5 & 0.1 & 1.1445E-05 & 5.5212E-06 & 2.6206E-06 & 1.2338E-06 & 1.07  \\
           &     & 0.5 & 9.1447E-07 & 4.0557E-07 & 1.6360E-07 & 6.2855E-08 & 1.29  \\
           & 0.8 & 0.1 & 9.0587E-06 & 4.4519E-06 & 2.1331E-06 & 1.0092E-06 & 1.06  \\
           &     & 0.5 & 4.2889E-07 & 2.5195E-07 & 1.1405E-07 & 4.6557E-08 & 1.07  \\
    \midrule
    CN-I   & 0.2 & 0.1 & 2.2824E-06 & 5.5842E-07 & 1.3804E-07 & 3.4301E-08 & 2.02  \\
           &     & 0.5 & 4.6121E-06 & 1.1336E-06 & 2.8101E-07 & 6.9956E-08 & 2.01  \\
           & 0.5 & 0.1 & 3.6320E-06 & 8.8802E-07 & 2.1938E-07 & 5.4478E-08 & 2.02  \\
           &     & 0.5 & 5.6160E-06 & 1.3795E-06 & 3.4186E-07 & 8.5088E-08 & 2.01  \\
           & 0.8 & 0.1 & 4.4374E-06 & 1.0837E-06 & 2.6755E-07 & 6.6422E-08 & 2.02  \\
           &     & 0.5 & 6.6373E-06 & 1.6290E-06 & 4.0349E-07 & 1.0040E-07 & 2.02  \\
    \bottomrule
\end{tabular}
\end{table}

\begin{table}[htb!]
  \centering
  \caption{Errors and convergence rates by schemes CN-JLZ \cite{JinLZ:2018a}, CN-WWY \cite{WangWY:2021} and scheme \eqref{eq:CN} for case (b) of Example \ref{exm:f0}.}\label{tab:f0b}
  \setlength{\tabcolsep}{8.0pt}\small
  \begin{tabular}{lccccccc}
    \toprule
    scheme & $\alpha$ & $\mu$ & $N=$80 & 160 & 320 & 640 & rate  \\
    \midrule
    CN-JLZ & 0.2 & 0.1 & 4.6760E-09 & 2.2521E-09 & 1.0678E-09 & 5.0244E-10 & 1.07  \\
           &     & 0.5 & 3.7747E-10 & 1.6711E-10 & 6.7183E-11 & 2.5742E-11 & 1.29  \\
           & 0.5 & 0.1 & 7.4765E-09 & 3.6628E-09 & 1.7513E-09 & 8.2752E-10 & 1.06  \\
           &     & 0.5 & 3.7941E-10 & 2.1640E-10 & 9.6281E-11 & 3.8891E-11 & 1.10  \\
           & 0.8 & 0.1 & 4.5037E-09 & 2.2486E-09 & 1.0849E-09 & 5.1499E-10 & 1.04  \\
           &     & 0.5 & 3.4909E-11 & 8.5516E-11 & 4.7854E-11 & 2.1245E-11 & 0.24  \\
    \midrule
    CN-WWY & 0.2 & 0.1 & 4.7159E-09 & 2.2616E-09 & 1.0702E-09 & 5.0301E-10 & 1.08  \\
           &     & 0.5 & 4.1742E-10 & 1.7665E-10 & 6.9515E-11 & 2.6318E-11 & 1.33  \\
           & 0.5 & 0.1 & 7.6834E-09 & 3.7122E-09 & 1.7634E-09 & 8.3050E-10 & 1.07  \\
           &     & 0.5 & 5.8638E-10 & 2.6583E-10 & 1.0836E-10 & 4.1876E-11 & 1.27  \\
           & 0.8 & 0.1 & 4.7534E-09 & 2.3083E-09 & 1.0995E-09 & 5.1860E-10 & 1.07  \\
           &     & 0.5 & 2.8458E-10 & 1.4517E-10 & 6.2437E-11 & 2.4850E-11 & 1.17  \\
    \midrule
    CN-I   & 0.2 & 0.1 & 2.3703E-08 & 5.8007E-09 & 1.4349E-09 & 3.5684E-10 & 2.02  \\
           &     & 0.5 & 6.3948E-08 & 1.5721E-08 & 3.8977E-09 & 9.7040E-10 & 2.01  \\
           & 0.5 & 0.1 & 2.4555E-08 & 6.0088E-09 & 1.4863E-09 & 3.6958E-10 & 2.02  \\
           &     & 0.5 & 6.4642E-08 & 1.5891E-08 & 3.9397E-09 & 9.8086E-10 & 2.01  \\
           & 0.8 & 0.1 & 2.4435E-08 & 5.9793E-09 & 1.4790E-09 & 3.6777E-10 & 2.02  \\
           &     & 0.5 & 6.4887E-08 & 1.5951E-08 & 3.9546E-09 & 9.8455E-10 & 2.01  \\
    \bottomrule
\end{tabular}
\end{table}

In Tables \ref{tab:f0a}-\ref{tab:f0b}, the numerical results by the Crank-Nicolson schemes CN-JLZ \cite{JinLZ:2018a}, CN-WWY \cite{WangWY:2021} and our proposed Crank-Nicolson scheme \eqref{eq:CN} are compared for 1D and 2D cases in Example \ref{exm:f0} with low regular source terms, where $\alpha=0.2,0.5,0.8$ and $\mu=0.1,0.5$. It indicates that the schemes CN-JLZ \cite{JinLZ:2018a} and CN-WWY \cite{WangWY:2021} lose the optimal second order convergence rate. However, our scheme \eqref{eq:CN} remains to keep the optimal second order convergence rate, which confirms the theoretical error estimate. In addition, our second Crank-Nicolson scheme \eqref{eq:nCN} also converges with second order for both cases in Example \ref{exm:f0}, the numerical results are omitted as much more results by \eqref{eq:nCN} are reported in the following examples with singular source terms.

\begin{example}\label{exm:f1}
  Consider the one dimensional problem \eqref{eq:tfdesub} with $T=1$, $\Omega=(0,1)$ and the following data:
  \begin{itemize}
    \item[\rm(a)] $u^0(x)\equiv0$ and $f(x,t)=\chi_{[0,\frac{1}{2}]}(t)t^{\mu}x^{-\frac{1}{4}}$ with $-1<\mu<0$, where $\chi_{[0,\frac{1}{2}]}(t)$ is the indicator function over the time interval $[0,\frac{1}{2}]$;
    \item[\rm(b)] $u^0(x)=\chi_{[\frac{1}{4},\frac{3}{4}]}(x)$ and $f(x,t)\equiv0$, where $\chi_{[\frac{1}{4},\frac{3}{4}]}(x)$ is the indicator function over $[\frac{1}{4},\frac{3}{4}]$.
  \end{itemize}
\end{example}

\begin{table}[htb!]
  \centering
  \caption{Errors and convergence rates by scheme \eqref{eq:nCN} for case (a) of Example \ref{exm:f1}.}\label{tab:f1a}
  \setlength{\tabcolsep}{8.0pt}\small
  \begin{tabular}{ccccccc}
    \toprule
    $\alpha$ & $\mu$ & $N=$80 & 160 & 320 & 640 & rate  \\
    \midrule
     0.1 & -0.1 & 1.2641E-06 & 3.0073E-07 & 7.3295E-08 & 1.8097E-08 & 2.04  \\
         & -0.5 & 2.3004E-06 & 5.4868E-07 & 1.3389E-07 & 3.3082E-08 & 2.04  \\
         & -0.9 & 8.9980E-06 & 2.1628E-06 & 5.2914E-07 & 1.3058E-07 & 2.04  \\
     0.5 & -0.1 & 9.2521E-06 & 2.1959E-06 & 5.3439E-07 & 1.3179E-07 & 2.04  \\
         & -0.5 & 1.6394E-05 & 3.9009E-06 & 9.5050E-07 & 2.3453E-07 & 2.04  \\
         & -0.9 & 5.9161E-05 & 1.4191E-05 & 3.4686E-06 & 8.5625E-07 & 2.04  \\
     0.9 & -0.1 & 2.7447E-05 & 6.4460E-06 & 1.5586E-06 & 3.8305E-07 & 2.05  \\
         & -0.5 & 4.4356E-05 & 1.0443E-05 & 2.5283E-06 & 6.2172E-07 & 2.05  \\
         & -0.9 & 1.1618E-04 & 2.7613E-05 & 6.7135E-06 & 1.6533E-06 & 2.04  \\
    \bottomrule
\end{tabular}
\end{table}

\begin{table}[htb!]
  \centering
  \caption{Errors and convergence rates by scheme \eqref{eq:nCN} for case (b) of Example \ref{exm:f1}.}\label{tab:f1b}
  \setlength{\tabcolsep}{10pt}\small
  \begin{tabular}{cccccc}
    \toprule
     $\alpha$ & $N=$80 & 160 & 320 & 640 & rate  \\
    \midrule
     0.1 & 1.7805E-06 & 4.3515E-07 & 1.0755E-07 & 2.6729E-08 & 2.02  \\
     0.5 & 8.7487E-06 & 2.1285E-06 & 5.2484E-07 & 1.3030E-07 & 2.02  \\
     0.9 & 8.4638E-06 & 2.0436E-06 & 5.0143E-07 & 1.2410E-07 & 2.03  \\
    \bottomrule
  \end{tabular}
\end{table}

The errors and convergence rates obtained by the Crank-Nicolson scheme \eqref{eq:nCN} for case (a) of Example \ref{exm:f1} are shown in Table \ref{tab:f1a} with $\alpha=0.1,0.5,0.9$ and $\mu=-0.1,-0.5,-0.9$. It is observed that our proposed Crank-Nicolson scheme \eqref{eq:nCN} converges with rate $O(\tau^2)$, which is consistent with our theoretical result and shows the effectiveness of the scheme \eqref{eq:nCN} for solving the problem \eqref{eq:tfdesub} with singular and nonsmooth source terms.
For case (b) of Example \ref{exm:f1}, the numerical results computed by the scheme \eqref{eq:nCN} are presented with $\alpha=0.1,0.5,0.9$ in Table \ref{tab:f1b}, which also verify the theoretical convergence result.

\begin{example}\label{exm:f2}
  Let $T=1$ and $\Omega=(0,1)^2$. Consider the two dimensional problem \eqref{eq:tfdesub} with the following data:
  \begin{itemize}
    \item[\rm(a)] $u^0(x)\equiv0$ and $f(x,t)=t^{\mu}\chi_{[\frac{1}{4},\frac{3}{4}]\times[\frac{1}{4},\frac{3}{4}]}(x)$, where $-1<\mu<0$ and $\chi_{[\frac{1}{4},\frac{3}{4}]\times[\frac{1}{4},\frac{3}{4}]}(x)$ is the indicator function over the space domain $[\frac{1}{4},\frac{3}{4}]\times[\frac{1}{4},\frac{3}{4}]$.
    \item[\rm(b)] $u^0(x)=\chi_{[\frac{1}{4},\frac{3}{4}]\times[\frac{1}{4},\frac{3}{4}]}(x)$ and $f(x,t)\equiv0$.
  \end{itemize}
\end{example}

\begin{table}[htb!]
\centering
\caption{Errors and convergence rates by the scheme \eqref{eq:nCN} for case (a) of Example \ref{exm:f2}.}\label{tab:f2a}
  \setlength{\tabcolsep}{8.0pt}\small
  \begin{tabular}{ccccccc}
    \toprule
     $\alpha$ & $\mu$ & $N=$80 & 160 & 320 & 640 & rate  \\
    \midrule
     0.1 & -0.2 & 1.2359E-07 & 3.0169E-08 & 7.4522E-09 & 1.8523E-09 & 2.02  \\
          & -0.5 & 3.8953E-07 & 9.4784E-08 & 2.3376E-08 & 5.8035E-09 & 2.02  \\
          & -0.8 & 7.5442E-07 & 1.8299E-07 & 4.5055E-08 & 1.1175E-08 & 2.03  \\
     0.5 & -0.2 & 1.2319E-07 & 3.0073E-08 & 7.4287E-09 & 1.8459E-09 & 2.02  \\
          & -0.5 & 3.9145E-07 & 9.5251E-08 & 2.3490E-08 & 5.8335E-09 & 2.02  \\
          & -0.8 & 7.6555E-07 & 1.8568E-07 & 4.5715E-08 & 1.1338E-08 & 2.03  \\
     0.9 & -0.2 & 1.2516E-07 & 3.0553E-08 & 7.5471E-09 & 1.8756E-09 & 2.02  \\
          & -0.5 & 3.9560E-07 & 9.6359E-08 & 2.3763E-08 & 5.9006E-09 & 2.02  \\
          & -0.8 & 7.6580E-07 & 1.8696E-07 & 4.6031E-08 & 1.1419E-08 & 2.02  \\
    \bottomrule
  \end{tabular}
\end{table}

\begin{table}[htb!]
  \centering
  \caption{Errors and convergence rates by the scheme \eqref{eq:nCN} for case (b) of Example \ref{exm:f2}.}\label{tab:f2b}
  \setlength{\tabcolsep}{10pt}\small
  \begin{tabular}{cccccc}
    \toprule
      $\alpha$ & $N=$80 & 160 & 320 & 640 & rate  \\
    \midrule
      0.1 & 5.2753E-08 & 1.2891E-08 & 3.1861E-09 & 7.9177E-10 & 2.02  \\
            0.5 & 2.2085E-07 & 5.3739E-08 & 1.3253E-08 & 3.2906E-09 & 2.02  \\
            0.9 & 9.8069E-08 & 2.3563E-08 & 5.7979E-09 & 1.4379E-09 & 2.03  \\
    \bottomrule
\end{tabular}
\end{table}

In Tables \ref{tab:f2a}-\ref{tab:f2b}, the numerical results obtained by the Crank-Nicolson scheme \eqref{eq:nCN} for two dimensional sub-diffusion problems \eqref{eq:tfdesub} in Example \ref{exm:f2} are shown, respectively. As the similar efficient performances for one dimensional problem in Example \ref{exm:f1}, the proposed Crank-Nicolson scheme \eqref{eq:nCN} also performs effectively and converges numerically with the optimal second order for the two dimensional sub-diffusion problem \eqref{eq:tfdesub} with singular source term.

\section{Conclusions}
In this paper, we develop two novel fractional Crank-Nicolson schemes without corrections for solving the sub-diffusion problem \eqref{eq:tfdesub} with nonsingular and singular source terms in time. We first propose a novel Crank-Nicolson scheme without corrections for the problem with regular source terms. Moreover, for problems with singular source terms, another fractional Crank-Nicolson scheme is designed and discussed in details. The error estimates of the two schemes are rigorously analyzed by the Laplace transform technique, and proved to be convergent with the optimal second-order for both nonsingular and singular source terms. The theoretical results are verified in the numerical examples.

%
%






\end{document}